\documentclass[11pt]{amsart}
\usepackage[foot]{amsaddr}
\usepackage[english]{babel}
\usepackage[
pdftex,hyperfootnotes]
{hyperref}
\usepackage[usenames,dvipsnames,svgnames,table,x11names]{xcolor}
\usepackage{pgfplots}
\usepackage{tikz}
\usepackage{tikz-3dplot}
\usetikzlibrary{calc,shadows,shapes.callouts,shapes.geometric,shapes.misc,positioning,patterns,%
	decorations.pathmorphing,decorations.markings,decorations.fractals,decorations.pathreplacing,shadings,fadings}
\pgfplotsset{compat=1.14} 
\usepackage[utf8]{inputenc}
\usepackage{nicematrix,booktabs}

\usepackage{comment}
\usepackage[textwidth=17.5cm,textheight=22.275cm,
height=
24.275cm,
width=18cm]{geometry}
\usepackage{rotating,multirow,pdflscape}
\usepackage{amssymb,latexsym,amsmath,amsthm}
\usepackage{mathrsfs}
\usepackage{mathtools}
\usepackage{bigints}
\usepackage{tensor}

\mathtoolsset{centercolon}
\usepackage[toc,page]{appendix}

\usepackage{tocvsec2}

\usepackage{Baskervaldx}
\newtheorem{theorem}{Theorem}
\newtheorem{pro}{ Proposition }

\newtheorem{rem}{Remark}

\renewcommand{\d}{\operatorname{d}}
\newcommand{\Exp}[1]{\operatorname{e}^{#1}}

\renewcommand{\Re}{\operatorname{Re}}
\newcommand{\diag}{\operatorname{diag}}

\newcommand{\I}{\mathbb{I}}
\newcommand{\C}{\mathbb{C}}

\newcommand{\N}{\mathbb{N}}
\newcommand{\R}{\mathbb{R}}

\makeatletter
\def\@settitle{\begin{center}%
		\baselineskip14\p@\relax
		\bfseries
		\uppercasenonmath\@title
		\@title
		\ifx\@subtitle\@empty\else
		\\[1ex]\uppercasenonmath\@subtitle
		\footnotesize\mdseries\@subtitle
		\fi
	\end{center}%
}
\def\subtitle#1{\gdef\@subtitle{#1}}
\def\@subtitle{}
\makeatother

\begin{document}
\title[Discrete Orthogonal Polynomials and Cholesky Factorization]{Pearson Equations for  Discrete Orthogonal Polynomials:\\I. Generalized Hypergeometric Functions and Toda Equations}

\author[M Mañas]{Manuel Mañas$^1$}
\address{$^1$Departamento de Física Teórica, Universidad Complutense de Madrid, Plaza Ciencias 1, 28040-Madrid, Spain \&
	Instituto de Ciencias Matematicas (ICMAT), Campus de Cantoblanco UAM, 28049-Madrid, Spain}
\email{$^1$manuel.manas@ucm.es}


\author[I Fernández-Irisarri]{Itsaso Fernández-Irisarri$^2$}
\email{$^2$itsasofe@ucm.es}

\author[OF González-Hernández]{Omar F. González-Hernández$^3$}
\email{$^3$omgonzal@math.uc3m.es}

\thanks{$^1$Thanks financial support from the Spanish ``Agencia Estatal de Investigación" research project [PGC2018-096504-B-C33], \emph{Ortogonalidad y Aproximación: Teoría y Aplicaciones en Física Matemática}.}

	

\begin{abstract}
The Cholesky factorization of the moment matrix is applied to discrete orthogonal polynomials on the homogeneous lattice. In particular, semiclassical discrete orthogonal polynomials, which are built in terms of a discrete Pearson equation, are studied. The Laguerre--Freud structure semi-infinite matrix  that models the shifts by $\pm 1$  in the independent variable of the set of orthogonal polynomials is introduced. In the semiclassical case it is proven that this Laguerre--Freud  matrix is  banded.  From the well known fact that moments of the semiclassical weights are logarithmic derivatives of generalized hypergeometric functions, it is shown how the contiguous relations  for these hypergeometric functions translate as symmetries for the corresponding moment matrix. It is found that the 3D Nijhoff--Capel  discrete Toda lattice describes the corresponding contiguous shifts for the squared norms of the orthogonal polynomials. The  continuous Toda  for these semiclassical discrete orthogonal polynomials is discussed and the compatibility equations are derived. It also shown that the Kadomtesev--Petvishvilii equation is  connected to an adequate deformed semiclassical discrete weight, but in this case the deformation do not satisfy a Pearson equation.
\end{abstract}

\maketitle

\tableofcontents

\section{Introduction}

In this paper, the Gauss--Borel factorization of the moment matrix approach to orthogonality  is applied to study discrete orthogonal polynomials subject to a Pearson equation.

Discrete orthogonal polynomials is nowadays a well established subject, the classical case has been treated extensively in    \cite{NSU} and the the Riemann--Hilbert problem has been used to study asymptotics  and applications, see  \cite{baik}.  
Semiclassical reductions, in where a discrete Pearson equation is fulfilled by the weight has been also  treated in the literature. See for example \cite{diego_paco,diego_paco1} and \cite{diego,diego1} and references therein for an comprehensive account. For some specific type of weights of generalized Charlier and Meixner types, the corresponding  Freud--Laguerre type equations for the coefficients of the three term recurrence has been studied, see for example \cite{clarkson,filipuk_vanassche0,filipuk_vanassche1,filipuk_vanassche2,smet_vanassche}. For a general account see \cite{Ismail,Ismail2,Beals_Wong} and for discrete orthogonal polynomials and Painlevé equations see \cite{walter}.

With this introduction we give a  fast briefing, of introductory character,  of orthogonal polynomials, its discrete version and the Cholesky factorization of the moment matrix. We also discuss for the first time in this context the so called Pascal matrices and its dressed versions.

Then, in \S 2   we discuss the discrete Pearson equation and give Theorem \ref{teo:symmetry_Gram},  the first main result of the paper, 
that describes a new symmetry of the moment matrix, direct consequence of the Pearson equation, in terms of the Pascal matrix. We then deduce the corresponding symmetry for the Jacobi matrix, see Proposition \ref{pro:Jacobi_Pearson}. Then, in Theorem \ref{teo:Laguerre-Freud}, the second main result of the paper,  we describe a banded semi-infinite matrix, that we called Laguerre--Freud structure matrix, that models the shift in the spectral variable. The name is rooted on the fact that this banded matrix encodes  Laguerre--Freud equations for the recurrence coefficients of the orthogonal polynomial sequence, see for example \cite{Fernandez-Irrisarri_Manas}. The number of non-trivial superdiagonals and subdiagonals of this matrix is determined by the Pearson equation. Several properties involving the Jacobi and Laguerre--Freud matrices are derived. Of particular interest are those of compatibility type.

As  the Pearson equation leads to generalized hypergeometric moments, we study the contiguous relations of these functions and its description as further symmetries of the moment matrix, see  Theorem \ref{pro:Hypergeometric relations}. In Theorem \ref{teo:Nijhoff-Capel} the squared norms of these discrete orthogonal polynomials are shown to be solutions of a well known discrete integrable multidimensional  lattice, known as the Nijhoff--Capel lattice, see \cite{Hietarinta,nijhoff}.  We also discuss the Toda hierarchy deformations, with only the first flow preserving the Pearson reduction. The compatibility with this first Toda flow and the Pearson equation is discussed in Proposition \ref{pro:compatibility Pearson Toda}. We also give the connection of deformations of these weights with solutions of the Kadomtsev--Petviashvilli (KP) equation.

\subsection{Linear functionals and orthogonal polynomials}Given a linear functional $\rho_z\in\C^*[z]$, here $\C^*[z]$ denotes the dual of the linear space of complex polynomials $\C[z]$,  the corresponding   moment  matrix is
 \begin{align*}
 G&=(G_{n,m}), &G_{n,m}&=\rho_{n+m}, &\rho_n&= \big\langle\rho_z,z^n\big\rangle, & n,m\in\N_0:=\{0,1,2,\dots\},
 \end{align*}
with $\rho_n$ the $n$-th moment of the linear functional $\rho_z$.
If the moment matrix is such that  all its truncations, which are Hankel matrices, $G_{i+1,j}=G_{i,j+1}$,
 \begin{align*}
 G^{[k]}=\begin{pNiceMatrix}
 G_{0,0}&\Cdots &G_{0,k-1}\\
 \Vdots & & \Vdots\\
 G_{k-1,0}&\Cdots & G_{k-1,k-1}
 \end{pNiceMatrix}=\begin{pNiceMatrix}[columns-width = 0.5cm]
 \rho_{0}&\rho_1&\rho_2&\Cdots &\rho_{k-1}\\
 \rho_1 &\rho_2 &&\Iddots& \rho_k\\
 \rho_2&&&&\Vdots\\
 \Vdots& &\Iddots&&\\[9pt]
 \rho_{k-1}&\rho_k&\Cdots &&\rho_{2k-2}
 \end{pNiceMatrix}
 \end{align*}
 are nonsingular; i.e. the Hankel determinants $\varDelta_k:=\det G^{[k]} $ do not cancel, $\varDelta_k\neq 0 $, $k\in\N_0$, then there exists monic polynomials
 \begin{align}\label{eq:polynomials}
 P_n(z)&=z^n+p^1_n z^{n-1}+\dots+p_n^n, & n&\in\N_0,
 \end{align}
 with $p^1_0=0$, such that the following  orthogonality conditions are fulfilled
 \begin{align*}
  \big\langle \rho, P_n(z)z^k\big\rangle &=0, & k&\in\{0,\dots,n-1\},&
   \big\langle \rho, P_n(z)z^n\big\rangle &=H_n\neq 0.
 \end{align*}
 Moreover, the set $\{P_n(z)\}_{n\in\N_0}$ is an orthogonal set of polynomials
$ \big\langle\rho,P_n(z)P_m(z)\big\rangle=\delta_{n,m}H_n,$ $n,m\in\N_0$.
 In this case, we have a  symmetric bilinear form 
$ \langle F, G\rangle_\rho:=\langle \rho, FG\rangle$,
 such that the moment matrix is the Gram matrix of this bilinear form and
 $	\langle P_n, P_m\rangle_\rho:=\delta_{n,m} H_n$.
 In this paper we will use both denominations, moment matrix or Gram matrix,  indistinctly to refer to the matrix $G$.
 
\subsection{The shift matrix $\Lambda$}In terms of the semi-infinite vector of monomials
\begin{align*}
\chi(z):=\begin{pNiceMatrix}
1\\z\\z^2\\\Vdots
\end{pNiceMatrix}
\end{align*}
the Gram matrix is written as $G=\left\langle\rho, \chi\chi^\top\right\rangle$, and it becomes evident this matrix  is symmetric, i.e., $G=G^\top$.
The vector of monomials $\chi$ is an eigenvector of the \emph{shift matrix}
\begin{align*}
\Lambda:=\left(\begin{NiceMatrix}[columns-width = auto]
0 & 1 & 0 &\Cdots&\\
\Vdots& \Ddots&\Ddots &\Ddots&\\
&&&&\\
&&&&
\end{NiceMatrix}\right)
\end{align*}
i.e., $\Lambda \chi=x\chi$.
From here it  immediately   follows that
$\Lambda G=G\Lambda^\top$,
i.e., the Gram matrix is a Hankel matrix, as we previously said.
The transposed matrix 
\begin{align*}
\Lambda^\top=\left(\begin{NiceMatrix}[columns-width = auto]
0 & \Cdots  &&&\\
1 & \Ddots &&&\\
0 &\Ddots &&&\\
\Vdots&\Ddots&&&\\
&&&&
\end{NiceMatrix}\right)
\end{align*}
satisfies
$\Lambda\Lambda^\top=I$ and $\Lambda^\top\Lambda=I-E_{0,0}$, 
with $(E_{i,j})_{k,l}=\delta_{i,k}\delta_{j,l}$, $i,j\in\N_0$.

\subsection{The Cholesky factorization of the Gram matrix}
Being the Gram matrix symmetric its Borel--Gauss factorization reduces to a Cholesky factorization
\begin{align}\label{eq:Cholesky}
G=S^{-1}HS^{-\top},
\end{align}
where $S$ is a lower unitriangular matrix that can be written as 
\begin{align*}
	S=\left(\begin{NiceMatrix}[columns-width = .5cm]
	1 & 0 &\Cdots &\\
	S_{1,0 } &  1&\Ddots&\\
	S_{2,0} & S_{2,1} & \Ddots &\\
	\Vdots & \Ddots& \Ddots& 
	\end{NiceMatrix}\right),
	\end{align*}
	and $H=\diag(H_0,H_1,\dots)$ is a  diagonal matrix, with $H_k\neq 0$, for $k\in\N_0$.
	The Cholesky  factorization  hold whenever the principal minors of the moment matrix; i.e., the Hankel determinants $\varDelta_k$,  do not cancel.
	
	The components $P_n(z)$ of the  semi-infinite vector of  polynomials
	\begin{align}\label{eq:PS}
	P(z):=S\chi(z),
	\end{align}
	are the monic orthogonal polynomials of the functional $\rho$.  Indeed, from the Cholesky factorization we know that
	\begin{align*}
\left\langle\rho, \chi\chi^\top\right\rangle=G=S^{-1}HS^{-\top}
	\end{align*}
	so that
$S	\left\langle\rho, \chi\chi^\top\right\rangle S^\top=H$
	and, consequently
$		\left\langle\rho, S\chi\chi^\top S^\top\right\rangle=H$
and we get $	\left\langle\rho, PP^\top \right\rangle=H$,  which recollects the orthogonality of the polynomials $\{P_n(z)\}_{n=0}^\infty$.

\subsection{The Jacobi matrix}Given this construction is natural to introduce the lower Hessenberg semi-infinite matrix
\begin{align}\label{eq:Jacobi}
J=S\Lambda S^{-1}
\end{align}
that has the vector $P(z)$ as eigenvector with eigenvalue $z$, i.e. $JP(z)=zP(z)$.
The Hankel condition of the Gram matrix $\Lambda G=G\Lambda^\top$ together with the Cholesky factorization leads to
$\Lambda S^{-1} H S^{-\top} =S^{-1} H S^{-\top} \Lambda^\top$,
or, equivalently,
$S\Lambda S^{-1} H  =H S^{-\top} \Lambda^\top S^{-\top}$,
i.e., 
\begin{align}\label{eq:symmetry_J}
J H=(JH)^\top =HJ^\top.
\end{align}
That is,  the Hessenberg matrix $JH$ is symmetric, thus being Hessenberg and symmetric we deduce that is tridiagonal. Hence, the Jacobi matrix $J$ given in \eqref{eq:Jacobi} reads
\begin{align*}
J=\left(\begin{NiceMatrix}[columns-width = 0.5cm]
\beta_0 & 1& 0&\Cdots& \\
\gamma_1 &\beta_1 & 1 &\Ddots&\\
0 &\gamma_2 &\beta_2 &1 &\\
\Vdots&\Ddots& \Ddots& \Ddots&\Ddots 
\end{NiceMatrix}\right)
\end{align*}
and the eigenvalue equation $JP=zP$ is a three term recursion relation 
\begin{align*}
zP_n(z)&=P_{n+1}(z)+\beta_n P_n(z)+\gamma_n P_{n-1}(z),
\end{align*}
that with the  initial conditions $P_{-1}=0$ and $P_0=1$ completely determines   the sequence of orthogonal polynomials $\{P_n(z)\}_{n\in\N_0}$ in terms of the recursion coefficients $\beta_n,\gamma_n$.
\begin{pro}
	The recursion coefficients, in terms of the Hankel determinants, are given by
	\begin{align}\label{eq:equations0}
	\beta_n&=p_n^1-p_{n+1}^1=-\frac{\tilde \varDelta_n}{\varDelta_n}+\frac{\tilde \varDelta_{n+1}}{\varDelta_{n+1}},&  \gamma_{n+1}&=\frac{H_{n+1}}{H_{n}}=\frac{\varDelta_{n+1}\varDelta_{n-1}}{\varDelta_n^2},& n\in\N_0,
	\end{align}
\end{pro}
\begin{proof}
	For $n\in\N:=\{1,2,\dots\}$, from  $J^\top=H^{-1}J H$, we get that $\gamma_n=\frac{H_n}{H_{n-1}}$. On the other hand, as $J=S\Lambda S^{-1}$ and recalling that  for the coefficients $S_{n,n-1}$ of the first subdiagonal of $S$ we have $S_{n,n-1}=p^1_n$, the first nontrivial leading coefficient of the monic polynomial $P_n$, we get $\beta_n=p^1_n-p^1_{n+1}$.
\end{proof}
It can be easily shown that the second kind functions  also satisfy  the previous three term recursion relation
\begin{align*}
zQ_n(z)&=Q_{n+1}(z)+\beta_nQ_n(z)+\gamma_n Q_{n-1}(z),
\end{align*}
but now with the initial conditions  $Q_{-1}=-H_{-1}$ and $Q_0=S_\rho$ (the Markov--Stieltjes transform of $\rho$), with $\beta_n=q^1_n-q^1_{n-1}$.

For future use we introduce the following diagonal matrices
\begin{align*}
 \gamma&:=\diag (\gamma_1,\gamma_2,\dots),&
\beta&:=\diag(\beta_0 ,\beta_{1},\dots)
\end{align*}
and
\begin{align*}
J_-&:=\Lambda ^\top \gamma, & J_+&:=\beta+\Lambda,
\end{align*}
so that we have the splitting
\begin{align*}
J=\Lambda^\top\gamma+\beta+\Lambda&=J_-+J_+.
\end{align*}
In general, given any semi-infinite matrix $A$, we will write $A=A_-+A_+$, where $A_-$ is a strictly lower triangular matrix and $A_+$ an upper triangular matrix. Moreover, $A_0$ will denote the diagonal part of  $A$.

\subsection{The lower Pascal matrix}
The lower Pascal matrix, built up of  binomial numbers, is defined by
\begin{align}\label{eq:Pascal_matrix}
B&=(B_{n,m}), & B_{n,m}&:= \begin{cases}
\displaystyle \binom{n}{m}, & n\geq m,\\
0, &n<m.
\end{cases}
\end{align}
so that
\begin{align}\label{eq:Pascal}
\chi(z+1)=B\chi(z).	
\end{align}
Moreover,
\begin{align*}
B^{-1}&=(\tilde B_{n,m}), & \tilde B_{n,m}&:= \begin{cases}
(-1)^{n+m}\displaystyle \binom{n}{m}, & n\geq m,\\
0, &n<m.
\end{cases}
\end{align*}
and
\begin{align*}
\chi(z-1)=B^{-1}\chi(z).	
\end{align*}
The lower Pascal matrix and its inverse are explicitly given by
\begin{align*}
B&=\left(\begin{NiceMatrix}[columns-width =auto]
		1&0&\Cdots\\
		1&1&0&\Cdots\\
		1&2&1&0&\Cdots&\\		
		1& 3 & 3&1 & 0&\Cdots\\
		1&4 & 6 & 4 & 1&0&\Cdots\\
		1& 5 & 10 &10 &5&1&0&\Cdots\\
		\Vdots & & & & & &\Ddots&\Ddots
	\end{NiceMatrix}\right),&
	B^{-1}&=\left(\begin{NiceMatrix}[r]
	1&0&\Cdots\\
	-1&1&0&\Cdots\\
	1&-2&1&0&\Cdots&\\		
	-1& 3 & -3&1 & 0&\Cdots\\
	1&-4 & 6 & -4 & 1&0&\Cdots\\
	-1& 5 & -10 &10 &-5&1&0&\Cdots\\
	\Vdots & & & & & &\Ddots&\Ddots
\end{NiceMatrix}\right).
\end{align*}
Let us introduce the lower unitriangular semi-infinite matrices, lets us refer to them as \emph{dressed Pascal matrices,}
\begin{align*}
\Pi&:=SBS^{-1}, & \Pi^{-1}&:=SB^{-1}S^{-1}, 
\end{align*}
that are connection matrices; i.e.,
\begin{align}\label{eq:PascalP}
P(z+1)&=\Pi P(z), & P(z-1)&=\Pi^{-1}P(z).
\end{align}
The lower Pascal matrix can be expressed in terms of its subdiagonal structure as follows
\begin{align*}
B^{\pm 1}&=I\pm\Lambda^\top D+\big(\Lambda^\top\big)^2D^{[2]}\pm\big(\Lambda^\top\big)^3D^{[3]}+\cdots,
\end{align*}
 where the diagonal matrices $D,D^{[k]}$, with $k\in\N$, ($D=D^{[1]}$) are given by
 \begin{align*}
 D&=\diag(1,2,3,\dots),&
 D^{[k]}&=\frac{1}{k}\diag\big(k^{(k)}, (k+1)^{(k)},(k+2)^{(k)}\cdots\big), 
 \end{align*}
 in terms of the falling factorials
$ x^{(k)}=x(x-1)(x-2)\cdots (x-k+1)$.
That is,  
\begin{align*}
D^{[k]}_n&=\frac{(n+k)\cdots (n+1)}{k}, & k&\in\N, & n&\in\N_0.
\end{align*}

The lower unitriangular factor can be also written  in terms of its subdiagonals
\begin{align*}
S=I+\Lambda^\top S^{[1]}+\big(\Lambda^\top\big)^2S^{[2]}+\cdots
\end{align*}
with 
$S^{[k]}=\diag \big(S^{[k]}_0, S^{[k]}_1,\dots\big)$ 
diagonal matrices.  From \eqref{eq:PS} is clear the following connection  between these subdiagonals coefficients and the coefficients of the orthogonal polynomials, given in \eqref{eq:polynomials}, holds
\begin{align}\label{eq:Sp}
	S^{[k]}_k=p^k_{n+k}.
\end{align}


We will use the \emph{shift operators} $T_\pm$ acting over the diagonal matrices as follows
\begin{align*}
T_-\diag(a_0,a_1,\dots)&:=\diag (a_1, a_2,\dots),&
T_+\diag(a_0,a_1,\dots)&:=\diag(0,a_0,a_1,\dots),
\end{align*}
where $T_-$ is the lowering shift operator and $T_+$ the raising shift operator over the diagonal matrices.
These shift operators have the following important properties, for any diagonal matrix $A=\diag(A_0,A_1,\dots)$
\begin{align}\label{eq:ladder_lambda}
	\Lambda A&=(T_-A)\Lambda,&   A	\Lambda &=\Lambda (T_+A), 	& A \Lambda^\top  &=\Lambda^\top(T_-A), &\Lambda^\top  A &= (T_+A)\Lambda ^\top.
\end{align}

\begin{rem}
	Notice that the standard notation, see \cite{NSU}, for the differences of a sequence $\{f_n\}_{n\in\N_0}$,
\begin{align*}
\Delta f_n&:=f_{n+1}-f_n, 
& n&\in\N_0,
\\ 
\nabla f_n&=f_n-f_{n-1}, &n&\in\N,
\end{align*}
and $\nabla f_0= f_0$, connects with the shift operators  by means of
\begin{align*}
T_-&=I+\Delta , & T_+&=I-\nabla.
\end{align*} 
\end{rem}

In terms of these shift operators we find
\begin{align}\label{eq:theDs}
	2D^{[2]}&=(T_-D)D, & 3D^{[3]}&=(T_-^2D)(T_-D)D=2(T_-D^{[2]})D=2D^{[2]}(T_-^2 D)
\end{align}

\begin{pro}\label{pro:Sinv}
The inverse matrix $S^{-1}$ of the matrix $S$  expands in terms of subdiagonals as follows
\begin{align*}
S^{-1}&=I+\Lambda^\top S^{[-1]}+\big(\Lambda^\top\big)^2 S^{[-2]}+\cdots.
\end{align*}
The subdiagonals $S^{[-k]}$ are explicitly given  in terms of  the subdiagonals  of $S$, the first few are
\begin{align*}
 S^{[-1]}&=-S^{[1]}, \\S^{[-2]}&=-S^{[2]}+(T_-S^{[1]})S^{[1]},\\
S^{[-3]}&=-S^{[3]}+(T_-S^{[2]})S^{[1]}
+(T_-^2S^{[1]})S^{[2]}
-(T_-^2 S^{[1]})
(T_-S^{[1]})S^{[1]},\\
S^{[-4]}&=\hspace*{-2pt}\begin{multlined}[t][.95\textwidth]
-S^{[4]}+(T_-S^{[3]})S^{[1]}
+(T_-^2S^{[2]})S^{[2]}-(T_-^2S^{[2]})(T_-S^{[1])})S^{[1]}+(T_-^3S^{[1]})S^{[3]}
\\
-(T_-^3S^{[1]})(T_-S^{[2])})S^{[1]}
-
(T_-^3S^{[1]})(T_-^2S^{[1])})S^{[2]}
+(T_-^3S^{[1]})(T_-^2S^{[1]})(T_-S^{[1])})S^{[1]}.
\end{multlined}
\end{align*}
\end{pro}

\begin{pro}
	The following  \emph{nonlocal} expressions for the polynomial coefficients in terms of the recursion coefficients  hold true
\begin{align}\label{eq:poly_coff_recursion}
	p^1_{n+1}&=-\sum_{k=0}^n\beta_k, & p^2_{n+1}=-\sum_{k=1}^n\gamma_{k}+ \sum_{0\leq l <k< n}\beta_k\beta_l.
\end{align}
Moreover, 
\begin{align}\label{eq:poly_recursion_coff_3}
	p^3_{n+2}-p^3_{n+3}=\gamma_{n+2}p^1_{n+1}+(\beta_{n+2}+\beta_{n+1}+\beta_n)p^2_{n+2}-
	(\beta_{n+1}+\beta_n)p^1_{n+2}p^1_{n+1}.
\end{align}
\end{pro}
\begin{proof}We have 
\begin{align*}
	J&=S\Lambda S^{-1}=(I+\Lambda^\top S^{[1]}+\big(\Lambda^\top\big)^2 S^{[2]}+\cdots)\Lambda
(	I+\Lambda^\top S^{[-1]}+\big(\Lambda^\top\big)^2 S^{[-2]}+\cdots)\\&
=\begin{multlined}[t][0.9\textwidth]
	\Lambda+T_+S^{[1]}+S^{[-1]}+ \Lambda^\top(T_+S^{[2]}+S^{[-2]}+S^{[1]}S^{[-1]})\\+(\Lambda^\top)^2\big(T_+S^{[3]}+S^{[-3]}+S^{[2]}S^{[-1]}
+(T_+S^{[1]})S^{[-2]}\big)+\cdots.
\end{multlined}
\end{align*}
Thus, we obtain
\begin{align}
\label{eq:Sbeta}	\beta&=T_+S^{[1]}-S^{[1]},\\
\label{eq:Sgamma}	\gamma &=T_+S^{[2]}-S^{[2]}+(T_-S^{[1]}-S^{[1]})S^{[1]},
	\end{align}
and an infinite set of relations among  subdiagonals of $S$, being the first
\begin{align}\label{eq:S3}
	T_+S^{[3]}+S^{[-3]}+S^{[2]}S^{[-1]}+(T_+S^{[1]})S^{[-2]}=0.
\end{align}
The relation \eqref{eq:Sbeta} component wise is\eqref{eq:equations0}.
As well, component wise,  the relation \eqref{eq:Sgamma} is
\begin{align*}
	\gamma_{n+1}&=S_{n-1}^{[2]}-S_n^{[2]}+(S_{n+1}^{[1]}-S_n^{[1]})S_n^{[1]}
	=p^2_{n+1}-p^2_{n+2}-\beta_{n+1} p^1_{n+1}.
\end{align*}
Hence, using telescoping series again, we find \eqref{eq:poly_coff_recursion}.
Finally, Equation  \eqref{eq:S3} reads
\begin{multline*}
	T_+S^{[3]}-S^{[3]}+(T_-S^{[2]})S^{[1]}
+(T_-^2S^{[1]})S^{[2]}
-(T_-^2 S^{[1]})
(T_-S^{[1]})S^{[1]}-S^{[2]}S^{[1]}
\\+(T_+S^{[1]})\big(-S^{[2]}+(T_-S^{[1]})S^{[1]}\big)=0,
\end{multline*}
that we can write
\begin{multline*}
	T_+S^{[3]}-S^{[3]}+T_-\big(S^{[2]}-T_+S^{[2]}-	(T_-S^{[1]}-S^{[1]})S^{[1]}\big)S^{[1]}
	+(T_-^2S^{[1]}-T_+S^{[1]})S^{[2]}
	\\+(T_+S^{[1]}-T_-S^{[1]})(T_-S^{[1]})S^{[1]}=0,
\end{multline*}
so that
\begin{align*}
	T_+S^{[3]}-S^{[3]}=(T_-\gamma )S^{[1]}
	+(T_-^2\beta+T_-\beta+\beta)S^{[2]}-(T_-\beta+\beta)(T_-S^{[1]})S^{[1]}=0,
\end{align*}
and component wise we have \eqref{eq:poly_recursion_coff_3}.
\end{proof}
\begin{rem}
In \eqref{eq:poly_recursion_coff_3} we can sum up on the LHS,  observe that is a telescoping series,  to get a nonlocal  nonlinear expression, in terms of the recursion coefficients, for $p^3_n$. A similar statement  holds for every  $p^k_n$ with $k=4,5,\dots$.
\end{rem}

Now, we turn our attention to the dressed Pascal matrix. 
We also expand the dressed Pascal matrices  into subdiagonals
\begin{align*}
\Pi^{\pm 1}=I+\Lambda^\top \pi^{[\pm 1]}+(\Lambda^\top)^2  \pi^{[\pm 2]}+\cdots
\end{align*}
with $\pi^{[\pm n]}=\diag(\pi^{[\pm n]}_0,\pi^{[\pm n]}_1,\dots)$. Then, for these subdiagonals we  find

\begin{pro}[The dressed Pascal matrix coefficients]
	We have
	\begin{gather}\label{eq:pis}
	\begin{aligned}
	\pi^{[\pm 1]}_n&=\pm (n+1), &
	\pi^{[\pm 2]}_n&=\frac{(n+2)(n+1)}{2}\pm
	p^1_{n+2}(n+1)\mp (n+2) p^{1}_{n+1}\\&&&=\frac{(n+2)(n+1)}{2}\mp (n+1)\beta_{n+1}
	\mp  p^{1}_{n+1},
	\end{aligned}\\\label{eq:piss}
	\begin{multlined}[t][\textwidth]
	\pi^{[\pm 3]}_n=\pm\frac{(n+3)(n+2)(n+1)}{3}+\frac{(n+2)(n+1)}{2}p^1_{n+3}-
	\frac{(n+3)(n+2)}{2}p^1_{n+1}\\\pm (n+1) p^2_{n+3}\mp (n+3)p^2_{n+2}\pm
	(n+3)p^1_{n+2}p^1_{n+1}\mp(n+2)p^1_{n+3}p^1_{n+1}.\end{multlined}
	\end{gather}
Moreover, the following relations are fulfill
	\begin{gather}\label{eq:pis2}
	\begin{aligned}
		\pi^{[1]}+\pi^{[-1]}&=0, &\pi^{[2]}+\pi^{[-2]}&=2D^{[2]}, &\pi^{[3]}+\pi^{[-3]}&=2((T_-^2S^{[1]})D^{[2]}-(T_-D^{[2]})S^{[1]}),
	\end{aligned}\\\notag
	\begin{aligned}
			\pi^{[1]}-\pi^{[-1]}&=2 D, &\pi^{[2]}-\pi^{[-2]}&=2((T_-S^{[1]})D- (T_-D) S^{[1]}),
	\end{aligned}\\\notag
\pi^{[3]}-\pi^{[-3]}= 2\big(D^{[3]}
+ (T_-S^{[2]})D- (T_-^2D)S^{[2]}+(T_-^2D)(T_-S^{[1]})S^{[1]}- (T_-^2S^{[1]})(T_-D)S^{[1]}\big).
	\end{gather}
\end{pro}

\begin{proof}
	From \eqref{eq:ladder_lambda} we get
	{\small\begin{align*}\hspace*{-2cm}
			\Pi^{\pm 1}&= \begin{multlined}[t][\textwidth]
				(I+\Lambda^\top S^{[1]}+\big(\Lambda^\top\big)^2S^{[2]}+\cdots )\big(I\pm\Lambda^\top D+\big(\Lambda^\top\big)^2D^{[2]}\pm\big(\Lambda^\top\big)^3D^{[3]}+\cdots\big) (I+\Lambda^\top S^{[-1]}+\big(\Lambda^\top\big)^2 S^{[-2]}+\cdots )
			\end{multlined}\\&=
			\begin{multlined}[t][\textwidth]I+(I+\Lambda^\top S^{[1]}+\big(\Lambda^\top\big)^2S^{[2]}+\cdots )\big(\pm\Lambda^\top D+\big(\Lambda^\top\big)^2D^{[2]}\pm\big(\Lambda^\top\big)^3D^{[3]}+\cdots\big)(I+\Lambda^\top S^{[-1]}+\big(\Lambda^\top\big)^2 S^{[-2]}+\cdots )\end{multlined}
	\end{align*}}
	so that
	{\small\begin{multline}\label{eq:Pi}
			\Pi^{\pm 1}=
			I\pm\Lambda ^\top D+\big(\Lambda^\top\big)^2\big(D^{[2]}\pm
			(T_-S^{[1]})D\pm (T_-D) S^{[-1]}\big)+
			(\Lambda^\top)^3\big(\pm D^{[3]}+(T_-^2S^{[1]})D^{[2]}+(T_-D^{[2]})S^{[-1]}
			\\\pm (T_-S^{[2]})D\pm (T_-^2D)S^{[-2]}\pm (T_-^2S^{[1]})(T_-D)S^{[-1]}\big)+\cdots.
	\end{multline}}
From \eqref{eq:Pi} and Proposition \ref{pro:Sinv} we  obtain
\begin{align}
\label{eq:pP1}\pi^{[\pm 1]}=&\pm D, \\
\label{eq:Pi2}
\pi^{[\pm 2]}=&D^{[2]}\pm
(T_-S^{[1]})D\mp (T_-D) S^{[1]},\\
\label{eq:Pi3}
\pi^{[\pm 3]}=&\begin{multlined}[t][.8\textwidth]\pm D^{[3]}+(T_-^2S^{[1]})D^{[2]}-(T_-D^{[2]})S^{[1]}
\pm (T_-S^{[2]})D\mp (T_-^2D)S^{[2]}\\ \pm (T_-^2D)(T_-S^{[1]})S^{[1]}\mp (T_-^2S^{[1]})(T_-D)S^{[1]}.
\end{multlined}
\end{align}
That component wise gives the desired result once we use   the expressions  for the $\beta$'s in \eqref{eq:equations0}.
\end{proof}


\begin{pro}
For any polynomial  $R(z)$ we have
\begin{align}
	\notag
	R(\Lambda)B^{\pm1}&=B^{\pm 1}R(\Lambda\pm I),&B^{\pm 1}R(\Lambda)&=R(\Lambda\mp I)B^{\pm 1},
	\\\label{eq:JPi}
R(J)\Pi^{\pm1}&=\Pi^{\pm 1}R(J\pm I), & \Pi^{\pm 1}R(J)&=R(J\mp I)\Pi^{\pm1}.
\end{align}
\end{pro}

\begin{proof}

		The compatibility condition of 
	\begin{align*}
\left\{\begin{aligned}
	B^{\pm 1}\chi(z)&=\chi(z\pm 1),\\
	R(\Lambda)\chi(z)&=R(z)\chi(z).
\end{aligned}\right.
	\end{align*}
	reads,
	\begin{align*}
R	(\Lambda) B^{\pm 1}\chi(z)=	R(\Lambda) \chi(z\pm 1)=R(z\pm 1)\chi(z\pm 1)=R(z\pm 1)B^{\pm 1}\chi (z)=B^{\pm 1}R(\Lambda\pm I)\chi(z),
\end{align*}
so that 	
$R(\Lambda)B^{\pm1}=B^{\pm 1}R(\Lambda\pm I)$ and, consequently, $B^{\pm 1}R(\Lambda)=R(\Lambda\mp I)B^{\pm 1}$. Finally, a similarity transformation  $\Lambda=S^{-1} J S$ gives the result.
\end{proof}
\section{Discrete orthogonal polynomials and discrete Pearson equations}

 \subsection{Discrete Pearson equation}
We now assume that the functional is a sum of Dirac delta functions supported $\N_0$, 
\begin{align*}
	\rho=\sum_{k=0}^\infty \delta(z-k) w(k).
\end{align*}
for some weight function $w(z)$ with finite values $w(k)$ for $k\in\N_0$. Hence, the bilinear form is
$\langle F, G\rangle=\sum_{k=0}^\infty F(k)G(k)w(k) $.
The moments are
\begin{align}\label{eq:moments}
\rho_n=\sum_{k=0}^\infty k^n w(k),
\end{align}
and, in particular,  the $0$-th moment reads as follows
\begin{align}\label{eq:first_moment}
\rho_0=\sum_{k=0}^\infty w(k).
\end{align}

We will study   families of weights are those 
 satisfying the following  \emph{discrete Pearson equation}
\begin{align}\label{eq:Pearson0}
\nabla (\sigma w)&=\tau w,
\end{align}
 that is $\sigma(k) w(k)-\sigma(k-1) w(k-1)=\tau(k)w(k)$, for $k\in\{1,2,\dots\}$, 
with $\sigma(z),\tau(z)\in\R[z]$ polynomials.
If we write  $\theta:=\sigma-\tau$, the previous Pearson equation reads
\begin{align}\label{eq:Pearson}
\theta(k+1)w(k+1)&=\sigma(k)w(k), &
k\in\N_0.
\end{align}
\begin{theorem}[Hypergeometric symmetry of the moment matrix]\label{teo:symmetry_Gram}
	Let the weight $w$ be subject to the discrete Pearson equation  \eqref{eq:Pearson}, where $\theta,\sigma$  are  polynomials with  $\theta(0)=0$. Then, the corresponding moment matrix fulfills
	\begin{align}\label{eq:Gram symmetry}
	\theta(\Lambda)G=B\sigma(\Lambda)GB^\top.
	\end{align}
\end{theorem}
\begin{proof}
	The moment matrix is
	\begin{align}\label{eq:Gram_discrete}
	G=\sum_{k=0}^\infty \chi(k)\chi(k)^\top w(k).
	\end{align}
	Thus,\enlargethispage{2cm}
		\begin{align*}
\theta(\Lambda)	G&=\sum_{k=0}^\infty \theta(\Lambda)	\chi(k)\chi(k)^\top w(k) &\text{use \eqref{eq:Gram_discrete}}\\
&=\sum_{k=1}^\infty 	\chi(k)\chi(k)^\top \theta(k)w(k) & \text{use $\Lambda\chi=z\chi$ and $\theta(0)=0$}\\
&=\sum_{k=0}^\infty 	\chi(k+1)\chi(k+1)^\top \theta(k+1)w(k+1) &\text{shift the summation variable} \\
&=\sum_{k=0}^\infty 	\chi(k+1)\chi(k+1)^\top \sigma(k)w(k)&\text{use Pearson equation \eqref{eq:Pearson}}\\
&=\sum_{k=0}^\infty 	B\chi(k)\chi(k)^\top B^\top\sigma(k)w(k) &\text{use \eqref{eq:Pascal}}\\
&=\sum_{k=0}^\infty 	B\sigma(\Lambda)\chi(k)\chi(k)^\top w(k)B^\top  & 
\text{use $\Lambda\chi=z\chi$ again}\\
&=B\sigma(\Lambda)GB^\top&
\text{use \eqref{eq:Gram_discrete}}.
	\end{align*}
\end{proof}

\begin{rem}
	This result extends to the case when $\theta$ and $\sigma$ are entire functions, not necessarily polynomials, and we can ensure some meaning for $\theta(\Lambda)$ and $\sigma(\Lambda)$. Later on,   see \S \ref{sec:hypergeometry and shifts}, we will show that this symmetry of the Gram matrix is a direct consequence of the generalized hypergeometric ODE satisfied by the the $0$-th moment.
\end{rem}

\subsection{Consequences for the  Jacobi matrix}
We can use the Cholesky factorization \eqref{eq:Cholesky} and the Jacobi matrix \eqref{eq:Jacobi} to get
\begin{pro}[Symmetry of the Jacobi matrix]\label{pro:Jacobi_Pearson}
	Let the weight $w$ be subject to the discrete Pearson equation \eqref{eq:Pearson}, where the functions $\theta,\sigma$  are entire functions, not necessarily polynomials, with  $\theta(0)=0$.  Then, 
	\begin{align}\label{eq:Jacobi symmetry}
\Pi^{-1}	H\theta(J^\top)=\sigma(J)H\Pi^\top
	\end{align}	
	Moreover,  $H\theta(J^\top)$ and $\sigma(J)H$ are symmetric matrices.
\end{pro}
\begin{proof}
	Given the Hankel property, $\Lambda G=G\Lambda^\top$, we write \eqref{eq:Gram symmetry} as 
$	G\theta(\Lambda^\top)=B\sigma(\Lambda)GB^\top$,
	and using the Cholesky factorization \eqref{eq:Cholesky} we get
$	S^{-1} HS^{-\top}\theta(\Lambda^\top)=
	B\sigma(\Lambda)S^{-1} H S^{-\top}B^\top$,
	so that
$ HS^{-\top}\theta(\Lambda^\top)S^\top=
	SBS^{-1} S\sigma(\Lambda)S^{-1} H S^{-\top}B^\top S^\top$,
	and we get Equation \eqref{eq:Jacobi symmetry}.

Let us prove that the matrices $H\theta(J^\top)$ and $\sigma(J)H$ are symmetric. This fact follows from \eqref{eq:symmetry_J}, $JH=HJ^\top$. Indeed,
	\begin{align*}
	(H\theta(J^\top))^\top&=\theta(J)H=\theta(HJ^\top H^{-1})H=H\theta(J^\top)H^{-1}H=H\theta(J^\top),\\
	(\sigma(J)H)^\top&=H\sigma(J^\top)=H\sigma(H^{-1}J H)=HH^{-1}\sigma(J)H=\sigma(J)H.
	\end{align*}
\end{proof}

\subsection{Generalized hypergeometric functions and the  Pearson equation}
Let us assume that $\theta,\sigma$ are polynomials,
and  denote their respective degrees by
$N+1:=\deg\theta(z)$ and $M:=\deg\sigma(z)$.
The roots of these polynomials are denoted by $\{-b_i+1\}_{i=1}^{N}$ and $\{-a_i\}_{i=1}^M$.  Following \cite{diego_paco} we choose 
\begin{align*}
\theta(z)&= z(z+b_1-1)\cdots(z+b_{N}-1), &
\sigma(z)&= \eta (z+a_1)\cdots(z+a_M).
\end{align*}
Notice that we have normalized $\theta$ to be a monic polynomial, while $\sigma$ is not monic, being the coefficient of the leading power denoted by $\eta$. This parameter $\eta$ will be instrumental in what follows.
Therefore,  see \cite{diego_paco}, we have that the weight satisfying  the Pearson equation \eqref{eq:Pearson} is proportional to
\begin{align}\label{eq:Pearson_weight}
w(z)=\frac{(a_1)_z\cdots(a_M)_z}{\Gamma(z+1)(b_1)_z\cdots(b_{N})_z}\eta^z,
\end{align}
where the Pochhammer symbol is
$(\alpha)_{z}=\dfrac {\Gamma (\alpha+z)}{\Gamma (\alpha)}$,
that for $z\in \N$ has the standard expression
\begin{align*}
(\alpha)_{n}&=\alpha(\alpha+1)(\alpha+2)\cdots (\alpha+n-1),& (\alpha)_{0}&=1.
\end{align*}
The moments of this weight are finite whenever, see \cite{diego_paco} and references therein,
\begin{enumerate}
	\item $M\leq N$ and $\eta\in\C $.
\item$M> N$, $\eta\in\C $ and one or more of the parameters $a_i$ is a nonpositive integer.
	\item $M=N+1$ and $|\eta|<1$.
	\item $M=N+1$, $|\eta|=1$ and 
	$\Re(b_1+\cdots+b_{N-1}-a_1-\dots -a_M)>0
	$.
\end{enumerate}

According to \eqref{eq:first_moment} the $0$-th moment  
\begin{align*}
\rho_0&=\sum_{k=0}^\infty w(k)=\sum_{k=0}^\infty \frac{(a_1)_k\cdots(a_M)_k}{(b_1+1)_k\cdots(b_{N}+1)_k}\frac{\eta^k}{k!}\\
&=\tensor[_M]{F}{_{N}} (a_1,\dots,a_M;b_1,\dots,b_{N};\eta)
=
{\displaystyle \,{}_{M}F_{N}\left[{\begin{matrix}a_{1}&\cdots &a_{M}\\b_{1}&\cdots &b_{N}\end{matrix}};\eta\right].}
\end{align*}
is the generalized hypergeometric function, where we are using the two standard 
notations,
see \cite{generalized_hypegeometric_functions}.
Then,  according to
\eqref{eq:moments}, for $n\in\N$, the corresponding  higher moments  $\rho_n=\sum_{k=0}^\infty k^n w(k)$, are
\begin{align}\label{eq:moment_n}
\rho_n&=\vartheta_\eta^n\rho_0=\vartheta_\eta^n\Big({\displaystyle \,{}_{M}F_{N}\left[{\begin{matrix}a_{1}&\cdots &a_{M}\\b_{1}&\cdots &b_{N}\end{matrix}};\eta\right]}\Big), &\vartheta_\eta:=\eta\frac{\partial }{\partial \eta}.
\end{align}

Given a function $f(\eta)$, we consider the Wronskian
\begin{align*}
\mathscr W_n(f)=\det\begin{pNiceMatrix}[columns-width = 0.5cm]
f &\vartheta_\eta f& \vartheta_\eta^2f&\Cdots &\vartheta_\eta^kf\\
	\vartheta_\eta f& \vartheta_\eta^2f&&\Iddots& \vartheta_\eta^{k+1}f\\
	\vartheta_\eta^2 f&&&&\\
	\Vdots& &\Iddots&&\\
	\vartheta_z^kf&\vartheta_\eta^{k+1}f& &&\vartheta_\eta^{2k}f
\end{pNiceMatrix}.
\end{align*}
Then, we have that the Hankel determinants $\varDelta_k=\det G^{[k]}$ determined by the truncations of the corresponding Gram  matrix are Wronskians of generalized hypergeometric functions
\begin{align}\label{eq:hankel_hyper1}
\varDelta_k&=\tau_k, & \tau_k&:=\mathscr W_{k}\Big({\displaystyle \,{}_{M}F_{N}\left[{\begin{matrix}a_{1}&\cdots &a_{M}\\b_{1}&\cdots &b_{N}\end{matrix}};\eta\right]}\Big),\\
\tilde \varDelta_k &=\vartheta_\eta\tau_k.\label{eq:hankel_hyper2}
\end{align}
We also have
\begin{align}\label{eq:Wp_n}
H_k&=\frac{\tau_{k+1}}{\tau_k},	&
p^1_k&=-\vartheta_\eta\log \tau_k.
\end{align}
\begin{rem}
	The functions $\tau_k$ are knwon in the theory of integrable systems as $\tau$-functions.
\end{rem}
\subsection{The Laguerre--Freud structure matrix and the Cholesky factorization}

\begin{theorem}[Laguerre--Freud structure matrix]\label{teo:Laguerre-Freud}
	Let us assume that the weight $w$ is subject to the discrete Pearson equation \eqref{eq:Pearson} with  $\theta,\sigma$ polynomials such that $\theta(0)=0$, $\deg\theta(z)=N+1$, $ \deg\sigma(z)=M$. 	Then, the Laguerre--Freud structure matrix
\begin{align}\label{eq:Psi}
\Psi&:=\Pi^{-1}H\theta(J^\top)=\sigma(J)H\Pi^\top=\Pi^{-1}\theta(J)H=H\sigma(J^\top)\Pi^\top\\
&=\theta(J+I)\Pi^{-1} H=H\Pi^\top\sigma(J^\top-I),\label{eq:Psi2}
\end{align}
has only  $N+M+2$ possibly nonzero  diagonals ($N+1$ superdiagonals, the main diagonal  and $M$ subdiagonals) 
\begin{align*}
\Psi=(\Lambda^\top)^M\psi^{(-M)}+\dots+\Lambda^\top \psi^{(-1)}+\psi^{(0)}+
\psi^{(1)}\Lambda+\dots+\psi^{(N+1)}\Lambda^{N+1},
\end{align*}
for some diagonal matrices $\psi^{(k)}$. In particular,  the lowest subdiagonal and highest superdiagonal  are given by
\begin{align}\label{eq:diagonals_Psi}
\left\{
\begin{aligned}
(\Lambda^\top)^M\psi^{(-M)}&=\eta(J_-)^MH,&
\psi^{(-M)}=\eta H\prod_{k=0}^{M-1}T_-^k\gamma=\eta\diag\Big(H_0\prod_{k=1}^{M}\gamma_k, H_1\prod_{k=2}^{M+1}\gamma_k,\dots\Big),\\
\psi^{(N+1)} \Lambda^{N+1}&=H(J_-^\top)^{N+1},&
\psi^{(N+1)}=H\prod_{k=0}^{N}T_-^k\gamma=\diag\Big(H_0\prod_{k=1}^{N+1}\gamma_k, H_1\prod_{k=2}^{N+2}\gamma_k,\dots\Big).
\end{aligned}
\right.
\end{align}
The vector $P(z)$ of orthogonal polynomials fulfill the following structure equations
\begin{align}\label{eq:P_shift}
\theta(z)P(z-1)&=\Psi H^{-1} P(z), &
\sigma(z)P(z+1)&=\Psi^\top H^{-1} P(z).
\end{align}
\end{theorem}

\begin{proof}
	To get \eqref{eq:Psi2} just use \eqref{eq:JPi}.
Now,
notice that the matrices   $H\theta(J^\top)$ and $\sigma(J)H$  
are banded matrices with $2N+3$ diagonals ($N+1$ superdiagonals and subdiagonals) and  $2M+1$ ($M$ superdiagonals and subdiagonals), respectively. Thus,  $\Pi^{-1}H\theta(J^\top)$  has at most $N+1$ superdiagonals while $\sigma(J)H\Pi^\top$ has at most $M$ subdiagonals. Consequently,  
$\Psi$ given by \eqref{eq:Psi}
is a banded semi-infinite matrix as described. 
Then, \eqref{eq:diagonals_Psi} follow from \eqref{eq:Psi} and the mentioned banded structure. To get the lowest subdiagonal we use $\Psi=\sigma(J)H\Pi^\top$, so that the lowest subdiagonal will come from the lowest subdiagonal of $\sigma(J)$, namely $\eta\, (J_-)^M$ right multiplied by the diagonal matrix $H$ and then right multiplied the main diagonal of $\Pi^\top$, which happens to be the identity.
To get the highest superdiagonal we proceed similarly and use
$\Psi=\Pi^{-1}H\theta(J^\top)$, so that the main superdiagonal will come from the product of the three factors $I$, $H$ and $\theta_N(J_-^\top)^{N+1}$, in this order. The component wise expressions are direct computations.

Finally, we have
\begin{align*}
&\begin{aligned}
\Psi H^{-1} P(z)&=\Pi^{-1}\theta(J)H H^{-1} P(z)\\
&=\Pi^{-1}\theta(J)P(z)\\
&=\Pi^{-1}\theta(z)P(z)\\
&=\theta(z)P(z-1),
\end{aligned}& &\begin{aligned}
\Psi^\top H^{-1} P(z)&=\Pi\sigma(J)H H^{-1} P(z)\\
&=\Pi\sigma (J)P(z)\\
&=\Pi\sigma(z)P(z)\\
&=\sigma(z)P(z+1),
\end{aligned}
 \end{align*}

\end{proof}

\begin{rem}[Laguerre--Freud equations]
	As we will see in the following sections \eqref{eq:Psi} will be instrumental in obtaining non linear recurrences for the recursion coefficients
\begin{align*}
\gamma_{n+1} &=F_1 (n,\gamma_n,\gamma_{n-1},\dots,\beta_n,\beta_{n-1}\dots), &
\beta_{n+1 }&= F_2 (n,\gamma_{n+1},\gamma_n,\dots,\beta_n,\beta_{n-1},\dots),
\end{align*}
for some functions $F_1,F_2$. These recurrences were named by Alphonse Magnus, attending to \cite{laguerre,freud}, as Laguerre--Freud see \cite{magnus,magnus1,magnus2,magnus3}.  This is the reason for the given name to $\Psi$. 
\end{rem}

 \begin{pro}
	The Laguerre--Freud  and Jacobi matrices fulfill
	\begin{align*}
	\sigma(J)\theta(J+I)&=\Psi H^{-1} \Psi^\top H^{-1}, &
	\theta(J)\sigma(J-I)&=\Psi^\top  H^{-1} \Psi H^{-1}.
	\end{align*}
	\begin{proof}
		We have
		\begin{align*}
	\sigma(J)\theta(J+I)P(z)&=	\sigma(z)\theta(z+1)P(z)=\sigma(z)\Psi H^{-1}P(z+1)=\Psi H^{-1}\sigma(z)P(z+1)\\&=
		\Psi H^{-1}\Psi^\top H^{-1} P(z),\\
		\theta(J)	\sigma(J-I)P(z)&=\theta(z)\sigma(z-1)P(z)=\theta(z)\Psi^\top H^{-1}P(z-1)=\Psi^\top  H^{-1}\theta(z)P(z-1)\\&=
		\Psi^\top H^{-1}\Psi H^{-1} P(z),
		\end{align*}
		that hold for all $z\in\C$ and, consequently, imply the desired result.
	\end{proof}
\end{pro}
\begin{theorem}[Cholesky factorization]
	Let us assume that $H\theta(J^\top)$ and $\sigma(J)H$ have the following Cholesky factorizations
	\begin{align}\label{eq:Choleskysigma}
	H\theta(J^\top)&= \Theta^{-1} H_\theta \Theta^{-\top},&
	\sigma(J)H&=\Sigma^{-1} H_\sigma \Sigma^{-\top}, 
	\end{align}
	with $\Theta$  and $\Sigma$  lower unitriangular matrices and  $H_\theta$ and $H_\sigma$   diagonals matrices. Then,  
	\begin{enumerate}
		\item  $\Theta^{-1}$ and  $\Sigma^{-1}$  have only  first   $N+1$ and  $M$ subdiagonals possibly nonzero, respectively. 
		\item We have
		\begin{align*}
		\Pi&=\Theta^{-1}\Sigma, &  H_\theta&=H_\sigma=:h.
		\end{align*}
		\item 
	The Laguerre--Freud matrix has the following Gauss--Borel factorization
			\begin{align*}
		\Psi= \Sigma^{-1} h \Theta^{-\top}.
		\end{align*}
	\end{enumerate}
\end{theorem}
\begin{proof}
Recall that according to Proposition \ref{pro:Jacobi_Pearson} the matrices  $H\theta(J^\top)$ and $\sigma(J)H$ are symmetric and, consequently, the corresponding Gauss--Borel factorizations, when they exists,  will be Cholesky factorizations.
	\begin{enumerate}
		\item  
		If follows from the fact that $J$ is tridiagonal and $\theta$ and $\sigma$ polynomials of degrees $N+1$ and $M$, respectively.
		\item 	The symmetry \eqref{eq:Jacobi symmetry} yields
		\begin{align*}
		\Pi^{-1}\Theta^{-1} H_\theta \Theta^{-\top}=\Sigma^{-1} H_\sigma \Sigma^{-\top}\Pi^\top
		\end{align*}
		and given the uniqueness of the Gauss--Borel factorization  the result follows.		
		\item From \eqref{eq:Psi} we have different alternatives to show the result, let us see two of them
		\begin{align*}
		\Psi&=\Pi^{-1}H\theta(J^\top)=\Sigma^{-1}\Theta \Theta^{-1} H_\theta \Theta^{-\top}=\Sigma^{-1} h\Theta^{-\top}, &
	\Psi&	=\sigma(J)H\Pi^\top=\Sigma^{-1} H_\sigma \Sigma^{-\top}
	\Sigma^\top  \Theta^{-1}=\Sigma^{-1} h\Theta^{-\top}.
		\end{align*}
	\end{enumerate}
\end{proof}

The compatibility with the recursion relation, i.e. eigenfunctions of the Jacobi matrix, and the recursion matrix leads to some interesting equations.
\begin{pro}
	The following compatibility conditions for the Laguerre--Freud  and Jacobi matrices  hold
\begin{subequations}
		\begin{align}\label{eq:compatibility_Jacobi_structure_a}
	[\Psi H^{-1},J]&=\Psi H^{-1}, \\ \label{eq:compatibility_Jacobi_structure_b}	[J, \Psi ^\top H^{-1}]&=\Psi ^\top H^{-1}.
	\end{align}
\end{subequations}
\end{pro}

\begin{proof}
To prove \eqref{eq:compatibility_Jacobi_structure_a}, we evaluate the eigenvalue equation $JP(z)=zP(z)$ in $z-1$ to get $JP(z-1)=(z-1)P(z-1)$. Now multiply it by $\theta (z)$ to obtain $J\theta(z)P(z-1)=(z-1)\theta(z)P(z-1)$. Therefore, recalling  \eqref{eq:P_shift},
	$J\Psi H^{-1}P(z)=(z-1)\Psi H^{-1}P(z)=\Psi H^{-1}(J-I)P(z)$ from where the relation follows.  
	
	Alternatively, observe that
	\begin{align}\label{eq:PsiBLambda}
		\Psi H^{-1}= S B^{-1}\theta(\Lambda) S^{-1}
	\end{align}
from where 
\begin{align*}
		[\Psi H^{-1}, J]&=S[B^{-1}\theta(\Lambda),\Lambda]S^{-1}=S[B^{-1},\Lambda]\theta(\Lambda)S^{-1}=
		SB^{-1}\theta(\Lambda)S^{-1}=\Psi H^{-1}.
\end{align*}

\enlargethispage{1cm}	
	To show \eqref{eq:compatibility_Jacobi_structure_b} we take the
	 transpose of the proved one to get 
	$[J^\top, H^{-1}\Psi ^\top ]=H^{-1}\Psi ^\top$ so that
\begin{align*}
		[HJ^\top H^{-1}, HH^{-1}\Psi ^\top H^{-1} ]=HH^{-1}\Psi ^\top H^{-1}
\end{align*}
	 and recalling that $HJ^\top H^{-1} =J$, see \eqref{eq:symmetry_J}, we get the desired result.
\end{proof}

\subsection{Contiguous  hypergeometric relations}
\label{sec:hypergeometry and shifts}

As we have seen, the polynomial discrete Pearson equation leads to \eqref{eq:moment_n}, and the Hankel determinants are Wronskians of a generalized  hypergeometric function, see \eqref{eq:hankel_hyper1} and \eqref{eq:hankel_hyper2}.
Hence, we expect that some properties of generalized hypergeometric functions  may translate to the Gram matrix.  To write  this dictionary we require of Proposition \ref{pro:Gram_eta},   discussed latter, that ensures,
	\begin{align*}
	\vartheta_\eta G=\Lambda G=G\Lambda^\top.
	\end{align*}

We now study three important relations fulfilled by the generalized hypergeometric function, namely:
\begin{align}
\label{eq:hyper1}\left(\vartheta_\eta+a_{i}\right){}_{M}F_{N}
\left[
\begin{gathered}
	\begin{alignedat}{5}
	a_{1}&&\cdots &&a_{i}&\cdots &&a_{M}
	\end{alignedat}\\[-5pt]
\begin{alignedat}{3}
	b_{1}&&\cdots &&b_{N}
\end{alignedat}
\end{gathered};\eta\right]&=a_{i}\;{}_{M}F_{N}\left[
\begin{gathered}
\begin{alignedat}{5}
a_{1}&&\cdots &&a_{i}+1&\cdots &&a_{M}
\end{alignedat}\\[-5pt]
\begin{alignedat}{3}
b_{1}&&\cdots &&b_{N}
\end{alignedat}
\end{gathered}
;\eta\right],\\
\label{eq:hyper2}
	\left(\vartheta_\eta+b_{j}-1\right){}_{M}F_{N}\left[
		\begin{gathered}
	\begin{alignedat}{3}
	a_{1}&&\cdots &&a_{M}
	\end{alignedat}\\[-5pt]
	\begin{alignedat}{5}
	b_{1}&&\cdots &&b_{j}&\cdots&b_{N}
	\end{alignedat}
	\end{gathered};\eta\right]&=(b_{j}-1)\;{}_{M}F_{N}\left[
	\begin{gathered}
\begin{alignedat}{3}
a_{1}&&\cdots &&a_{M}
\end{alignedat}\\[-5pt]
\begin{alignedat}{5}
b_{1}&&\cdots &&b_{j}-1&\cdots&b_{N}
\end{alignedat}
\end{gathered}
	;\eta\right],\\
	\label{eq:hyper3}
	{\frac {\rm {d}}{{\rm {d}}\eta}}\;{}_{M}F_{N}\left[{\begin{matrix}a_{1}&\cdots &a_{M}\\b_{1}&\cdots &b_{N}\end{matrix}};\eta\right]&=\kappa\;{}_{M}F_{N}\left[{\begin{matrix}a_{1}+1&\cdots &a_{M}+1\\b_{1}+1&\cdots &b_{N}+1\end{matrix}};\eta\right], &\kappa:={\frac {\prod _{i=1}^{M}a_{i}}{\prod _{j=1}^{N}b_{j}}}.
\end{align}
From these three equations we also derive
\begin{align}\label{eq:edo_hyper}
\eta\prod _{n=1}^{M}\left(\eta{\frac {\rm {d}}{{\rm {d}}\eta}}+a_{n}\right)u&=\eta{\frac {\rm {d}}{{\rm {d}}\eta}}\prod _{n=1}^{N}\left(\eta{\frac {\rm {d}}{{\rm {d}}\eta}}+b_{n}-1\right)u, & u:={}_{M}F_{N}\left[{\begin{matrix}a_{1}&\cdots &a_{M}\\b_{1}&\cdots &b_{N}\end{matrix}};\eta\right].
\end{align}
In \eqref{eq:hyper1} and \eqref{eq:hyper2} we have basic relations between contiguous generalized hypergeometric functions and its derivatives.

For the analysis of these equations let us introduce the shift operators  in the parameters
$\{a_i\}_{i=1}^M$ and $\{b_j\}_{j=1}^N$. Thus, given a function $f\left[{\begin{smallmatrix}a_{1}&\cdots &a_{M}\\b_{1}&\cdots &b_{N}\end{smallmatrix}}\right]$ of these parameters we introduce the shifts ${}_i T$ and $T_j$ as follows
\begin{align*}
{}_i T f\left[\begin{gathered}
\begin{alignedat}{5}
a_{1}&&\cdots &&a_{i}&\cdots &&a_{M}
\end{alignedat}\\[-5pt]
\begin{alignedat}{3}
b_{1}&&\cdots &&b_{N}
\end{alignedat}
\end{gathered}\right]&=f\left[\begin{gathered}
\begin{alignedat}{5}
a_{1}&&\cdots &&a_{i}+1&\cdots &&a_{M}
\end{alignedat}\\[-5pt]
\begin{alignedat}{3}
b_{1}&&\cdots &&b_{N}
\end{alignedat}
\end{gathered}\right],&
T_jf\left[	\begin{gathered}
\begin{alignedat}{3}
a_{1}&&\cdots &&a_{M}
\end{alignedat}\\[-5pt]
\begin{alignedat}{5}
b_{1}&&\cdots &&b_{j}&\cdots&b_{N}
\end{alignedat}
\end{gathered}\right]&=f\left[	\begin{gathered}
\begin{alignedat}{3}
a_{1}&&\cdots &&a_{M}
\end{alignedat}\\[-5pt]
\begin{alignedat}{5}
b_{1}&&\cdots &&b_{j}-1&\cdots&b_{N}
\end{alignedat}
\end{gathered}\right]
\end{align*}
and a total shift  $T={}_1T\cdots {}_MT \,T_1^{-1}\cdots T_N^{-1}$; i.e,
\begin{align*}
Tf\left[{\begin{matrix}a_{1}&\cdots &a_{M}\\b_{1}&\cdots &b_{N}\end{matrix}}\right]:=f\left[{\begin{matrix}a_{1}+1&\cdots &a_{M}+1\\b_{1}+1&\cdots &b_{N}+1\end{matrix}}\right].
\end{align*}
Then, we find:
\begin{theorem}[Hypergeometric relations]\label{pro:Hypergeometric relations}
	The moment matrix $G=(\rho_{n+m})_{n,n\in\N_0}$ of a weight satisfying \eqref{eq:Pearson} fulfills the following hypergeometric relations
\begin{subequations}\label{eq:Gram_hyper}
		\begin{align}
\label{eq:Gram_hyper1}	(\Lambda+a_iI)G&=a_i \;{}_i T G,\\
\label{eq:Gram_hyper2}		(\Lambda+(b_j -1)I)G&=(b_j-1)T_jG,\\
\label{eq:Gram_hyper3}		\Lambda G &=\kappa B (T G) B^\top
	\end{align}
\end{subequations}
	with  $B$ given in \eqref{eq:Pascal}.
\end{theorem}

\begin{proof}
	We first prove \eqref{eq:Gram_hyper1}. For that aim we use that $G_{n,m}=\vartheta_\eta^{n+m}\big({}_MF_N\left[{\begin{smallmatrix}a_{1}&\cdots &a_{M}\\b_{1}&\cdots &b_{N}\end{smallmatrix}};\eta\right]\big)$ so that
	\begin{align*}
	(\vartheta_\eta+a_i)G_{n,m}&=(\vartheta_\eta+a_i)\vartheta_\eta^{n+m}\Big({}_MF_N\left[{\begin{matrix}a_{1}&\cdots &a_{M}\\b_{1}&\cdots &b_{N}\end{matrix}};\eta\right]\Big)\\
	&=\vartheta_\eta^{n+m}(\vartheta_\eta+a_i)\Big({}_MF_N\left[{\begin{matrix}a_{1}&\cdots &a_{M}\\b_{1}&\cdots &b_{N}\end{matrix}};\eta\right]\Big)\\
	&=a_{i}\vartheta_\eta^{n+m}\Big({}_{M}F_{N}\left[\begin{gathered}
	\begin{alignedat}{5}
	a_{1}&&\cdots &&a_{i}+1&\cdots &&a_{M}
	\end{alignedat}\\[-5pt]
	\begin{alignedat}{3}
	b_{1}&&\cdots &&b_{N}
	\end{alignedat}
	\end{gathered};\eta\right]\Big)\\
	&=a_i\;{}_iTG_{n,m}.
	\end{align*}
	Thus, $(\vartheta_\eta+a_i)G=a_i \;{}_iTG$, and recalling $	\vartheta_\eta G=\Lambda G$ we get the result.  Relation \eqref{eq:Gram_hyper2} is proved similarly. Finally, \eqref{eq:Gram_hyper2} follows from \eqref{eq:hyper3}, the novelty is that in \eqref{eq:hyper3} we have $\frac{\d}{\d \eta}=\eta^{-1}\vartheta_\eta$, which do not commute with $\vartheta_\eta$. Observe that the relation
	\begin{align*}
\frac{\d}{\d \eta}\eta-\eta \frac{\d}{\d \eta}=1
	\end{align*}
	can be written
$	\eta^{-1}\vartheta_\eta\eta =\vartheta_\eta+1$.
	Hence,
		$	\eta^{-1}\vartheta_\eta^n\eta =(\eta^{-1}\vartheta_\eta\eta )^n=(\vartheta_\eta+1)^n$,
	that, in turn, implies
			\begin{align}\label{eq:vartheta_der}
	\frac{\d}{\d \eta}\vartheta_\eta^{n} =(\vartheta_\eta+1)^n \frac{\d}{\d \eta},
	\end{align}
	Thus,
	\begin{align*}
	\frac{\d G_{n,m}}{\d \eta}&=\frac{\d }{\d \eta}\vartheta_\eta^{n+m}\Big({}_MF_N\left[{\begin{matrix}a_{1}&\cdots &a_{M}\\b_{1}&\cdots &b_{N}\end{matrix}};\eta\right]\Big)=(\vartheta_\eta+1)^{n+m}\frac{\d }{\d \eta}\Big({}_MF_N\left[{\begin{matrix}a_{1}&\cdots &a_{M}\\b_{1}&\cdots &b_{N}\end{matrix}};\eta\right]\Big)\\	&=(\vartheta_\eta+1)^{n+m}
	\kappa\;{}_{M}F_{N}\left[{\begin{matrix}a_{1}+1&\cdots &a_{M}+1\\b_{1}+1&\cdots &b_{N}+1\end{matrix}};\eta\right]\\
	\\&=\kappa
\sum_{k=0}^{n+m}
\binom{n+m}{k}\vartheta_\eta^k{}_{M}F_{N}\left[{\begin{matrix}a_{1}+1&\cdots &a_{M}+1\\b_{1}+1&\cdots &b_{N}+1\end{matrix}};\eta\right].
	\end{align*}
	Using the  Chu--Vandermonde identity
\begin{align*}
\binom{n+m}{k}=\sum_{l=0}^{k}\binom{n}{l}\binom{m}{k-l},
\end{align*}
we find
	\begin{align*}
\frac{\d G_{n,m}}{\d \eta}&=\kappa
\sum_{k=0}^{n+m}\sum_{l=0}^{k}\binom{n}{l}\binom{m}{k-l}
\vartheta_\eta^{l+k-l}{}_{M}F_{N}\left[{\begin{matrix}a_{1}+1&\cdots &a_{M}+1\\b_{1}+1&\cdots &b_{N}+1\end{matrix}};\eta\right]\\
&=\kappa
\sum_{k=0}^{n+m}\sum_{l=0}^{k}\binom{n}{l}\binom{m}{k-l}
TG_{l,k-l}
\\
&=\kappa
\sum_{l=0}^{n}\sum_{k=0}^{m}\binom{n}{l}
TG_{l,k}\binom{m}{k}
\end{align*}
from where the result follows.\enlargethispage{2cm}

Finally,  using $\vartheta_\eta-1=\eta\vartheta_\eta\eta^{-1}$, we get $(\vartheta_\eta-1)^n\eta=\eta(\vartheta_\eta)^n$ and 
we  workout \eqref{eq:edo_hyper} as follows
\begin{align*}%
\eta\prod _{n=1}^{M}\left(\vartheta_\eta+a_{n}\right) G_{n,m}&=\eta\prod _{n=1}^{M}\left(\vartheta_\eta+a_{n}\right)\vartheta_\eta^{n+m}\Big({}_{M}F_{N}\left[{\begin{matrix}a_{1}&\cdots &a_{M}\\b_{1}&\cdots &b_{N}\end{matrix}};\eta\right]\Big)\\&=(\vartheta_\eta-1)^{n+m}\eta\prod _{n=1}^{M}\left(\vartheta_\eta+a_{n}\right)\Big({}_{M}F_{N}\left[{\begin{matrix}a_{1}&\cdots &a_{M}\\b_{1}&\cdots &b_{N}\end{matrix}};\eta\right]\Big)
\\&=\vartheta_\eta\prod _{n=1}^{N}\left(\vartheta_\eta+b_{n}-1\right)(\vartheta_\eta-1)^{n+m}\Big({}_{M}F_{N}\left[{\begin{matrix}a_{1}&\cdots &a_{M}\\b_{1}&\cdots &b_{N}\end{matrix}};\eta\right]\Big)\\&=
\vartheta_\eta\prod _{n=1}^{N}\left(\vartheta_\eta+b_{n}-1\right)\sum_{k=0}^{n+m}\binom{n+m}{k}(-1)^k\vartheta_\eta^k\Big({}_{M}F_{N}\left[{\begin{matrix}a_{1}&\cdots &a_{M}\\b_{1}&\cdots &b_{N}\end{matrix}};\eta\right]\Big)\\&=
\vartheta_\eta\prod _{n=1}^{N}\left(\vartheta_\eta+b_{n}-1\right)\sum_{l=0}^{n}\sum_{k=0}^{m}(-1)^l\binom{n}{l}
G_{l,k}(-1)^k\binom{m}{k},
\end{align*}
and, consequently, we finally deduce
\begin{align*}
\sigma(\Lambda) G=B^{-1}\theta(\Lambda)GB^{-\top},
\end{align*}
and result follows at once.
\end{proof}

\begin{rem}
From \eqref{eq:edo_hyper} we derive, in an alternative manner, the relation	\eqref{eq:Gram symmetry}.
\end{rem}

\subsection{Discrete lattice Toda equation }
The shifts in the hypergeometric parameters induce corresponding transformations on the discrete orthogonal polynomials. In order to describe them we introduce the following  semi-infinite matrices
	\begin{align*}
	\tensor[_i]{\Omega}{}&:=S(\tensor[_i]{T}{}S)^{-1}, & i&\in\{1,\dots,M\},&\Omega_k&:=S\,(T_kS)^{-1},& k&\in\{1,\dots,N\},
	\end{align*}
so that the following connection formulas are fulfilled
\begin{align}
	\label{eq:connection1}
	\tensor[_i]{\Omega}{}\hspace*{4pt}\tensor[_i]{T}{}P(z)&=P(z), & i&\in\{1,\dots,M\},\\
	\label{eq:connection2}
	\Omega_j \,T_jP(z)&=P(z),& j&\in\{1,\dots,N\}.
\end{align}
In \cite{Manas} we proved that
	\begin{align*}
			\tensor[_i]{\Omega}{}&=
			\left(\begin{NiceMatrix}[columns-width = 0.5cm,]
				1 &0 &\Cdots&\\
				\frac{1}{a_i} \frac{H_1}{{}_iTH_0} & 1& \Ddots\\
				0 &\frac{1}{a_i} \frac{H_2}{{}_iTH_1} & \Ddots&\\
				\Vdots &\Ddots &\Ddots & 
			\end{NiceMatrix}\right),&
			\Omega_j&=
			\left(\begin{NiceMatrix}[columns-width = 0.5cm,]
				1 &0 &\Cdots&\\
				\frac{1}{b_j-1} \frac{H_1}{T_jH_0} & 1& \Ddots\\
				0 &\frac{1}{b_j-1} \frac{H_2}{T_jH_1} & \Ddots&\\
				\Vdots &\Ddots &\Ddots & 
			\end{NiceMatrix}\right),
	\end{align*}
The connection formulas  \eqref{eq:connection1}	 and \eqref{eq:connection2}	
for contiguous hypergeometric parameters imply a nonlinear compatibility equations that leads to a generalized lattice Toda equation found in \cite{nijhoff}. We use the following notation 
\begin{align*}
	u_n=H_{n-1}\left[{\begin{smallmatrix}a_{1}&\cdots &a_{M}\\b_{1}&\cdots &b_{N}\end{smallmatrix}};\eta\right]
	\end{align*}
We now consider three  variables $(n,r,s)$, where $r,s$ are any couple of variables taken from the set of hypergeometric parameters $\{a_1,\dots,a_M,b_1,\dots,b_N\}$ and denote
the corresponding ``shifts'' in $n,r,s$ as follows 
\begin{align*}
	\bar u_n(r,s) &=u_{n+1}(r,s),  &\hat u_n(r,s)&=\hat a\hat u_n(r\pm1,s), &\tilde u_n(r,s)&=\tilde a  u(r,s\pm1),
\end{align*}
where the $\pm$ signs is a $+$ if the corresponding variable belongs to $\{a_i\}_{i=1}^M$ or a $-$ if it belongs to $\{b_j\}_{j=1}^M$, and  the constants $\hat a,\tilde a$ are taken from $\{a_i\}_{i=1}^N$ and $\{b_j-1\}_{j=1}^N$ according to the choice selected for the variables $r$ and $s$. We denote by $\Omega^{(r)}$ and $\Omega^{(s)}$, the corresponding  matrices $\Omega$'s  taken from  $\{{}_i\Omega\}_{i=1}^M$ and $\{\Omega_j\}_{j=1}^N$, depending on the variables picked  $r$ and $s$. 

\begin{theorem}[Nijhoff--Capel
	 discrete lattice Toda equations]\label{teo:Nijhoff-Capel}
	The squared norms satisfy the following nonlinear equations linking  contiguous hypergeometric parameters
	\begin{align}\label{eq:Nijhoff-Capel}
		\frac{\hat{\bar u} -
			\tilde{\bar u}}{\bar u}=\tilde{\hat u}\Big(	\frac{1}{\tilde u} -	\frac{1}{\hat u} \Big).
	\end{align} 
\end{theorem}
\begin{proof}
The connection formulas are
\begin{align*}
	\Omega^{(r)}\hat P&= P, & 	\Omega^{(s)}\tilde P&= P. 
\end{align*}
Then, we have
\begin{align*}
	\tilde{\Omega}^{(r)}\tilde{\hat P}=\tilde P, 
\end{align*}
so that
\begin{align*}
\Omega^{(s)}	\tilde{\Omega}^{(r)}\tilde{\hat P}=\Omega^{(s)}\tilde P=P, 
\end{align*}
and the compatibility $\tilde{\hat P}=\hat{ \tilde P}$ leads to the nonlinear condition
\begin{align}\label{eq:compatibility_contiguous}
		\Omega^{(s)}	\tilde{\Omega}^{(r)}= 	\Omega^{(r)}	\hat{\Omega}^{(s)}.
\end{align}
Then, as we can write
\begin{align*}
	\Omega^{(r)}&=
\left(\begin{NiceMatrix}[columns-width = 0.2cm]
	1 &0 &\Cdots&\\
\frac{\bar u_1}{\hat u_1} & 1& \Ddots\\
	0 & \frac{\bar u_2}{\hat u_2} & \Ddots&\\
	\Vdots &\Ddots &\Ddots & 
\end{NiceMatrix}\right),& 	
\Omega^{(s)}&=
\left(\begin{NiceMatrix}[columns-width = 0.2cm]
1 &0 &\Cdots&\\
\frac{\bar u_1}{\tilde u_1} & 1& \Ddots\\
0 &\frac{\bar u_2}{\tilde u_2} & \Ddots&\\
\Vdots &\Ddots &\Ddots & 
\end{NiceMatrix}\right),
\end{align*}
Equation \eqref{eq:compatibility_contiguous} is
\begin{align*}
\left(\begin{NiceMatrix}[columns-width = 0.2cm]
	1 &0 &\Cdots&\\
	\frac{\bar u_1}{\tilde u_1} & 1& \Ddots\\
	0 &\frac{\bar u_2}{\tilde u_2} & \Ddots&\\
	\Vdots &\Ddots &\Ddots & 
\end{NiceMatrix}\right)\left(\begin{NiceMatrix}[columns-width = 0.2cm]
1 &0 &\Cdots&\\
\frac{\tilde{\bar u}_1}{\tilde{\hat u}_1} & 1& \Ddots\\
0 & \frac{\tilde{\bar u}_2}{\tilde{\hat u}_2} & \Ddots&\\
\Vdots &\Ddots &\Ddots & 
\end{NiceMatrix}\right)=\left(\begin{NiceMatrix}[columns-width = 0.2cm]
1 &0 &\Cdots&\\
\frac{\bar u_1}{\hat u_1} & 1& \Ddots\\
0 & \frac{\bar u_2}{\hat u_2} & \Ddots&\\
\Vdots &\Ddots &\Ddots & 
\end{NiceMatrix}\right)\left(\begin{NiceMatrix}[columns-width = 0.2cm]
1 &0 &\Cdots&\\
\frac{\hat{\bar u}_1}{\hat{\tilde u}_1} & 1& \Ddots\\
0 &\frac{\hat{\bar u}_2}{\hat{\tilde u}_2} & \Ddots&\\
\Vdots &\Ddots &\Ddots & 
\end{NiceMatrix}\right),
\end{align*}
and, consequently, we deduce
\begin{align*}
	\frac{\bar u}{\tilde u} +	\frac{\tilde{\bar u}}{\tilde{\hat u}}=\frac{\bar u}{\hat u} +\frac{\hat{\bar u}}{\hat{\tilde u}}, 
\end{align*}
and the result follows immediately.
\end{proof}

\begin{rem}
	Equation \eqref{eq:Nijhoff-Capel} is the generalized discrete lattice Toda equation described for the first time by Nijhoff and Capel \cite{nijhoff}, that taking adequate continuous limits in the $r$ and $s$ variables (hypergeometric parameters) recovers the 2D Toda equation. This is a canonical equation among the difference equations in three variables involving up to second order differences  of octahedral type, consistent on the 4D lattice  \cite{adler}, it appears as the type V 
of the octahedron type integrable discrete equations in \S 3.9 of  the book \cite{Hietarinta} .  
\end{rem}

Now we study  compatibility of the recursion relation and the connection formula for contiguous parameters; i.e., the compatibility of
\begin{align*}
	zP&=JP,\\
\hat P&=	\hat a\frac{\omega^{(r)}}{z+r} P.
\end{align*}
These pair of relations lead  to
\begin{align*}
	\hat J\hat P(z)&=z\hat P (z)& &\Rightarrow &\hat J \frac{\omega^{(r)}}{z+r} P(z)&= z\frac{\omega^{(r)}}{z+r}P(z)=\frac{\omega^{(r)}}{z+r}JP(z),
\end{align*}
and therefore we find the  following compatibility equation
\begin{align*}
	\hat J \omega^{(r)}&=\omega^{(r)}J.
\end{align*}

Using the notation
\begin{align*}
	u_n&=H_{n-1}, & v_n&=\beta_{n-1}
\end{align*} 
so that $\gamma_n=\frac{\bar u_n}{u_n}$ and we have the following expressions
	\begin{align*}
		J&=\left(\begin{NiceMatrix}[columns-width = auto]
			v_1 & 1& 0&\Cdots& \\
			\frac{\bar u_1}{u_1} &v_2 & 1 &0&\Cdots\\
			0 &\Ddots &\Ddots &\Ddots &\Ddots\\
			\Vdots&\Ddots& & & 
		\end{NiceMatrix}\right),&
\omega^{(r)}&=	\left(\begin{NiceMatrix}[columns-width = auto]
	\frac{\hat u_1}{ u_1} &1 &0 &\Cdots\\
		0&  \frac{\hat u_2}{ u_2} &1&\Ddots\\
		\Vdots & \Ddots &\Ddots & \Ddots
	\end{NiceMatrix}\right),
\end{align*}
 the compatibility reads
\begin{align*}
\hspace*{-1cm}\left(\begin{NiceMatrix}[columns-width = 0.1cm]
\hat	v_1 & 1& 0&\Cdots& \\
	\frac{\hat{\bar u}_1}{\hat u_1} &\hat v_2 & 1 &0&\Cdots\\
	0 &\Ddots &\Ddots &\Ddots &\Ddots\\
	\Vdots&\Ddots& & & 
\end{NiceMatrix}\right)\left(\begin{NiceMatrix}[columns-width = auto]
\frac{\hat u_1}{ u_1} &1 &0 &\Cdots\\
0&  \frac{\hat u_2}{ u_2} &1&\Ddots\\
\Vdots & \Ddots &\Ddots & \Ddots
\end{NiceMatrix}\right)=\left(\begin{NiceMatrix}[columns-width = auto]
\frac{\hat u_1}{ u_1} &1 &0 &\Cdots\\
0&  \frac{\hat u_2}{ u_2} &1&\Ddots\\
\Vdots & \Ddots &\Ddots & \Ddots
\end{NiceMatrix}\right)\left(\begin{NiceMatrix}[columns-width =  0.1cm]
v_1 & 1& 0&\Cdots& \\
\frac{\bar u_1}{u_1} &v_2 & 1 &0&\Cdots\\
0 &\Ddots &\Ddots &\Ddots &\Ddots\\
\Vdots&\Ddots& & & 
\end{NiceMatrix}\right)
\end{align*}
and  we find $\hat v_1\frac{\hat u_1}{u_1}=v_1\frac{\hat u_1}{u_1}+\frac{\bar u_1}{u_1}$ and
\begin{align*}
	\frac{\hat{\bar u}_n}{\hat u_n}
	+\hat{\bar v}_n\frac{\hat{\bar u}_n}{\bar u_n}&=	
	\frac{\bar{\bar u}_n}{\bar  u_n}+ \bar v_n\frac{\hat{\bar u}_n}{\bar u_n},\\
			\hat{v}_n+\frac{\hat{\bar u}_n}{\bar u_n}&=\bar v_n+\frac{\hat u_n}{u_n}
\end{align*}
for $ n\in\N$.
Thus, we find

\begin{pro}
The squared norms $u_n=H_{n-1}$ and the recursion coefficients $v_n=\beta_{n-1}=p^1_{n-1}-p^1_n$, satisfy the following system of nonlinear   difference equations
\begin{align*}
	\hat{\bar v}-\bar v&=\frac{\bar{\bar u}}{\hat{\bar u}}-\frac{{\bar u}}{\hat u},\\
\hat v-\bar v&=\frac{\hat u}{u}-\frac{\hat {\bar u}}{\bar u}.
\end{align*}
\end{pro}
	
Finally, recalling $\vartheta_{\eta}P=\Phi P$ and $\Omega \hat P=P$,  we get 
\begin{align*}
	\vartheta_{\eta}\Omega=\Phi\Omega-\Omega \hat{\Phi},
\end{align*}
so that
\begin{pro}
	The squared norms $u_n=H_{n-1}$ fulfill
	\begin{align*}
		\vartheta_{\eta}\Big(\frac{\bar u}{\hat u}\Big)=\frac{\bar u}{u}-\frac{\hat{\bar u}}{\hat u}.
	\end{align*}
\end{pro}

\subsection{The Toda flows}

Given a semi-infinite vector  $\boldsymbol{\eta }:=\{\eta_l\}_{l=1}^\infty\in \C^\infty$, for each $z\in\C$ we define $\mathscr E(z;\boldsymbol{\eta}):=
\prod_{l=1}^\infty\eta_l^{z^l}$ and assume we assume a weight of the form $w(z)=v(z)\mathscr E(z;\boldsymbol{\eta})$,  with $v$ $\eta$-independent,  so that the corresponding moment matrix is
\begin{align*}
G=\sum_{k=0}^\infty \chi(k)\chi(k)^\top v(k)\mathscr E(k;\boldsymbol{\eta}).
\end{align*}
Observe  that only this first flow  preserves the Pearson reduction. Notice also that $\mathscr E=\Exp{\sum_{l=1}^\infty t_l z^l}$, with $t_l:=\log\eta_l$ the standard time flows in an integrable hierarchy. We will write $\eta_1=\eta$, normally $t_1$ is denoted by $x$, $t_1=x$. 
We will use 
\begin{align*}
\vartheta_l:=\eta_l\frac{\partial}{\partial \eta_l}=\frac{\partial}{\partial t_l}.
\end{align*}
In particular $\vartheta_1=\vartheta_\eta$.  To ensure the converge of the corresponding series for the moments we require $|\eta_k|\leq1$ for $k\in\{2,3,\dots\}$.

If we take $\eta_k=1$, $k\in\N$, then the corresponding Gram matrix is
\begin{align*}
G_0=\sum_{k=0}^\infty \chi(k)\chi(k)^\top v(k).
\end{align*}
Moreover, we can write for the deformed Gram matrix
\begin{align*}
G=\mathscr E(\Lambda;\boldsymbol \eta) G_0&=G_0\mathscr E(\Lambda^\top;\boldsymbol \eta) , &
\mathscr E(\Lambda;\boldsymbol \eta) &:=\prod_{l=1}^\infty\eta_l^{\Lambda^l}=\exp\big(\sum_{l=1}^\infty t_l \Lambda^l\big).
\end{align*}

	\begin{pro}\label{pro:Gram_eta}
	For a functional of the form $\rho=\sum\limits_{k=0}^\infty\delta(z-k)v(z)\mathscr E(z;\boldsymbol{\eta})$, given $l\in\N$,  the Gram matrix satisfies
	\begin{align}\label{eq:symmetrty_eta}
	\vartheta_l G=(\Lambda)^lG=G(\Lambda^\top)^l.
	\end{align}
\end{pro}
\begin{proof}
	It  follows form $\vartheta_{l}\mathscr E(z;\boldsymbol{\eta})=z^l\mathscr E(z;\boldsymbol{\eta})$, indeed
	\begin{align*}
	\vartheta_{l }G=\sum_{k=0}^\infty \chi(k)\chi(k)^\top k^lv(k)\mathscr E(k;\boldsymbol{\eta})=\Lambda^lG=G(\Lambda^\top)^l.
	\end{align*}
\end{proof}

For $k\in\N$, let us define the strictly lower triangular matrix (strictly because it has zeros on the main diagonal)
\begin{align*}
\Phi_k &:=(\vartheta_{k} S ) S^{-1}.
\end{align*}
In particular, we denote $\Phi:=\Phi_1=(\vartheta_{\eta} S ) S^{-1}$.

\begin{pro}
	The semi-infinite vector  $P$ fulfills
	\begin{align}\label{eq:MP}
	\vartheta_{k} P=\Phi_k  P.
	\end{align}
\end{pro}
\begin{proof}
	As $P=S\chi$ then $\vartheta_{k} P=(\vartheta_{k} S)\chi=(\vartheta_{k} S) S^{-1} S\chi=\Phi P$.
\end{proof}

\begin{pro}[Sato--Wilson equations]
	The following equations holds
	\begin{align}\label{eq:Mscr}
	-\Phi_k  H+\vartheta_{k} H-H \Phi_k ^\top&=J^kH,& k&\in\N.
	\end{align}
	Consequently, for $k\in\N$ and $n\in\N_0$ we have 
	\begin{align*}
	\Phi_k&=-(J^k)_-, &  \vartheta_k \log H_n=(J^k)_{n,n}.
		\end{align*}
\end{pro}

\begin{proof}
	Using  the Cholesky factorization \eqref{eq:Cholesky} for the moment  matrix we get
	\begin{align*}
	\vartheta_{k} G=\vartheta_{k}\big(S^{-1} H S^{-\top}\big)=-S^{-1}(\vartheta_{k} S) S^{-1} H S^{-\top}+ S^{-1} (\vartheta_{k} H) S^{-\top}- S^{-1} HS^{-\top}(\vartheta_{k} S)^\top S^{-\top}
	\end{align*}
	
	Hence, from the symmetry relation \eqref{eq:symmetrty_eta} we deduce
	\begin{align*}
	-S^{-1}(\vartheta_\eta S) S^{-1} H S^{-\top}+ S^{-1} (\vartheta_\eta H) S^{-\top}- S^{-1} HS^{-\top}(\vartheta_\eta S)^\top S^{-\top}=\Lambda^k S^{-1} H S^{-\top},
	\end{align*}
	from where
	\begin{align*}
	-(\vartheta_k S) S^{-1} H+ \vartheta_k H- H\big((\vartheta_k S) S^{-1}\big)^\top =J^kH ,
	\end{align*}
for $J=S\Lambda S^{-1}$,	and the result follows.
\end{proof}

For $k=1$; i.e. for the first flow $\eta_1=\eta$, we have formulas \eqref{eq:hankel_hyper1}, \eqref{eq:hankel_hyper2} and \eqref{eq:Wp_n}
in terms of the $\tau$-function. Moreover,

\begin{pro}[Toda]
The	 following  equations hold 
\begin{subequations}
		\begin{align}\label{eq:MS}
	\Phi =(\vartheta_{\eta} S)S^{-1}&=-\Lambda^\top\gamma,\\\label{eq:MS0}
	(\vartheta_\eta H) H^{-1}&=\beta.
	\end{align}
\end{subequations}
	In particular,  for $n,k-1\in\N$, we have
\begin{subequations}
		\begin{gather}\label{eq:equations}
\begin{aligned}
		\vartheta_\eta p^1_n&=-\gamma_n, & \vartheta_{\eta}p^k_{n+k}&=-\gamma_{n+k}p^{k-1}_{n+k-1}, 
\end{aligned}\\
	\vartheta_\eta \log H_n=\beta_ n.\label{eq:equationsH}
	\end{gather}
\end{subequations}
	The functions $q_n:=\log H_n$, $n\in\N$, satisfy the Toda equations
	\begin{align}\label{eq:Toda_equation}
	\vartheta_\eta^2q_n=\Exp{q_{n+1}-q_n}-\Exp{q_n-q_{n-1}}.
	\end{align}
	For $n\in\N$, we also have
\begin{align*}
\vartheta_\eta P_{n}(z)=-\gamma_n P_{n-1}(z).
\end{align*}	
\end{pro}

\begin{proof}
Equation	\eqref{eq:MS} is a direct consequence of \eqref{eq:Mscr} for $k=1$. To obtain \eqref{eq:equations} we
we write $\vartheta_{\eta}S=-\Lambda^\top\gamma S$ in terms of its subdiagonals, i.e.
\begin{align*}
	\vartheta_{\eta}\big(I+\Lambda^\top S^{[1]}+(\Lambda^\top )^2S^{[2]}+\cdots\big)=-\Lambda^\top\gamma\big(I+\Lambda^\top S^{[1]}+(\Lambda^\top )^2S^{[2]}+\cdots\big)
\end{align*}
that, for $k\in\N$,  gives 
\begin{align*}
\vartheta_\eta	S^{[k]}=-(T_-^{k-1}\gamma)	S^{[k-1]}
\end{align*}
with $S^{[0]}=I$. Now, recalling \eqref{eq:Sp} we get \eqref{eq:equations}. Equation \eqref{eq:equationsH}  follows  component wise from \eqref{eq:MS0}.

As $\beta_n=p^1_n-p^1_{n+1}$ and $\gamma_n=\frac{H_{n}}{H_{n-1}}$ we deduce that
	\begin{align*}
	\vartheta_\eta^2 \log H_n=\vartheta_\eta p^1_n-\vartheta_\eta p^1_{n+1}=-\frac{H_{n}}{H_{n-1}}+\frac{H_{n+1}}{H_{n}},
	\end{align*}
that is the Toda equation \eqref{eq:Toda_equation} for $H_n=\Exp{q_n}$.
	The last equation follows form $\vartheta_\eta  P=\Phi P=-\Lambda^\top \gamma P$.
\end{proof}

Given the $\tau$-function \eqref{eq:hankel_hyper1} we find
\begin{pro}[$\tau$-function expressions]
	In terms of the $\tau$-function 
	\begin{align*}
		\tau_n&:=\mathscr W_{n}\Big({\displaystyle \,{}_{M}F_{N}\left[{\begin{matrix}a_{1}&\cdots &a_{M}\\b_{1}&\cdots &b_{N}\end{matrix}};\eta\right]}\Big),
	\end{align*}
	we find
	\begin{align}\label{eq:toda_relations}
	H_n&=\frac{\tau_{n+1}}{\tau_n},	&
p^1_n&=-\vartheta_\eta\log \tau_n, &\gamma_n&=\vartheta_\eta^2\log\tau_n, &\beta_n&=\vartheta_{\eta}\log\frac{\tau_{n+1}}{\tau_n}.
\end{align}
\end{pro}
 \begin{proof}
Collect together \eqref{eq:Wp_n} and \eqref{eq:equations}.
 \end{proof}


\begin{pro}
	The following Lax equation holds
	\begin{align*}
	\vartheta_\eta J=[J_+,J].
	\end{align*}
	The  recursion coefficients satisfy the following Toda system
\begin{subequations}\label{eq:Toda_system}
		\begin{align}\label{eq:Toda_system_beta}
	\vartheta_\eta\beta_n&=\gamma_{n+1}-\gamma_n,\\\label{eq:Toda_system_gamma}
	\vartheta_\eta\log\gamma_n&=\beta_{n}-\beta_{n-1},
	\end{align}
\end{subequations}
	for $n\in\N_0$ and $\beta_{-1}=0$. Consequently, we get 
	\begin{align}\label{eq:Toda_equation_gamma}
	\vartheta_\eta^2\log\gamma_n+2\gamma_n&=\gamma_{n+1}+\gamma_{n-1}.
	\end{align}
	Moreover,
	\begin{align*}
	\vartheta_\eta^2\log\tau_n=\frac{\tau_{n+1}\tau_{n-1}}{\tau_n^2}.
	\end{align*}
\end{pro}

\begin{proof}
	As $J=S\Lambda S^{-1}$ is clear that, 
	\begin{align*}
	\vartheta_\eta J=[\Phi,J]=[-J_-,J]=[J_+-J,J]=[J_+,J]
	\end{align*}
	From this Lax equation one could derive, component-wise, the Toda system \eqref{eq:Toda_system}. Alternatively, the Toda system \eqref{eq:Toda_system} could be derived as follows. From \eqref{eq:equations0}  and \eqref{eq:equations} we obtain 
	\begin{align*}
	\vartheta_\eta\beta_n&=\vartheta_\eta	p^1_n-\vartheta_\eta p^1_{n+1}=-\gamma_n+\gamma_{n+1},\\
	\vartheta_\eta\log \gamma_n&=\vartheta_\eta\log H_n-\vartheta_\eta\log H_{n-1}=\beta_n-\beta_{n-1}.
	\end{align*}
Finally,
	\begin{align*}
	\vartheta_\eta^2\log\tau_n=\gamma_n=\frac{H_{n}}{H_{n-1}}=\frac{\tau_{n+1}\tau_{n-1}}{\tau_n}.
	\end{align*}
\end{proof}

For higher Toda flows, that generically are not consistent with the hypergeometric reduction,  we have
\begin{pro}[Lax equations]
	The following Lax equation holds
	\begin{align*}
	\vartheta_k J=[(J^k)_+,J].
	\end{align*}
\end{pro}
\begin{proof}
		As $J=S\Lambda S^{-1}$ we have
\begin{align*}
\vartheta_k J=[\Phi_k,J]=[-(J)^k_-,J]=[(J^k)_+-J^k,J]=[(J^k)_+,J]
\end{align*}
\end{proof}


\begin{pro}[Wave functions and Zakharov--Shabat equations]
	The wave function 
	satisfies the linear system
	\begin{align*}
	\vartheta_k \varPsi=(J^k)_+\varPsi.
	\end{align*}
	Moreover, the following Zakharov--Shabat equations hold
	\begin{align*}
	\vartheta_k (J^l)_+-\vartheta_l(J^k)_++[(J^l)_+,(J^k)_+]=0.
	\end{align*}
\end{pro}

\begin{proof}
The orthogonal polynomials we have
	\begin{align*}
\vartheta_{k} P=\Phi_k  P=((J^k)_+-J^k)P=((J^k)_+-z^k)P
\end{align*}
Thus, the wave function 
fulfills the following linear system
\begin{align*}
\vartheta_k \varPsi=\vartheta_k(P)\mathscr E+P\vartheta_k(\mathscr E)=(J^k)_+\varPsi.
\end{align*}
Finally, the Zakharov--Shabat equations follow as the compatibility conditions for the previous linear system.
\end{proof}
	
\begin{pro}
Let us assume that $G_0$, the Gram matrix of the functional	$\rho=\sum\limits_{k=0}^\infty\delta(z-k)v(z)$, admits a Cholesky factorization, and that  two semi-infinite matrices $Z_1$ and 
$Z_2$ are given, such that 
\begin{enumerate}
	\item $Z_1\mathscr E(\Lambda,\boldsymbol \eta)^{-1}$ is strictly lower triangular,
	\item $Z_2$ is upper triangular and
	\item $Z_1  (\boldsymbol{\eta}) G_0=Z_2  (\boldsymbol{\eta}) $.
\end{enumerate}
Then, we can ensure that $Z_1=Z_2=0$.
\end{pro}
\begin{proof}
We have\begin{align*}
	Z_1 G_0=Z_1 \mathscr E(\Lambda;\boldsymbol \eta)^{-1} G=Z_1 \mathscr E(\Lambda;\boldsymbol \eta)^{-1} S^{-1} H S^{-\top}
\end{align*}
and we get
\begin{align*}
Z_1 \mathscr E(\Lambda;\boldsymbol \eta)^{-1} S^{-1} H S^{-\top}=Z_2.
\end{align*}
Consequently,
\begin{align*}
Z_1 \mathscr E(\Lambda;\boldsymbol \eta)^{-1} S^{-1} =Z_2 S H^{-1}.
\end{align*}
But $Z_1 \mathscr E(\Lambda;\boldsymbol \eta)^{-1} S^{-1} $ is strictly lower triangular while $Z_2 S H^{-1}$ is upper triangular. As both terms are equal the only possibility is that both are the zero matrix
\begin{align*}
Z_1 \mathscr E(\Lambda;\boldsymbol \eta)^{-1} S^{-1} =Z_2 S H^{-1}=0,
\end{align*}
and we get the result.
\end{proof}

Using this results one proves that
\begin{pro}[KP equation]
	The wave function $\varPsi_n$ satisfies the linear equations
	\begin{align*}
	\vartheta_m\varPsi_n=\mathscr P_m(\vartheta_\eta)\varPsi_n
	\end{align*}
	with
\begin{align*}
	\mathscr P_m(\vartheta_\eta)=\vartheta_{\eta}^m-mp^1_n\vartheta_{\eta}^{m-2}+
U_{n,m-3}	\vartheta_{\eta}^{m-3}+ \cdots+U_{n,0}
\end{align*}
	where $\{U_{n,j}\}_{j=0}^{m-3}$ are polynomials in  $\vartheta_{\eta}^lp^1_n,\vartheta_2p^1_n,\dots,\vartheta_{m-1}p^1_n$.
	In particular, 
	the following linear equations hold
\begin{align*}
\vartheta_2\varPsi_n&=\vartheta_\eta^2\varPsi_n-2\vartheta_{\eta}(p^1_n)\varPsi_n,\\
\vartheta_3\varPsi_n &=\vartheta_\eta^3\varPsi_n-3\vartheta_{\eta}(p^1_n)\vartheta_{\eta}\varPsi_n
	-\frac{3}{2}\Big(\vartheta_{\eta}^2(p^1_k)+\vartheta_2(p^1_k)\Big)\varPsi_n,
\end{align*}
and, consequently, its compatibility condition,  which is the KP equation
	\begin{align*}
\vartheta_{\eta}\Big(
4\vartheta_3(p^1_n)+6\big(\vartheta_{\eta}(p^1_n)\big)^2-\vartheta_{\eta}^3(p^1_n)\Big)=\vartheta_2^2(p^1_n),
	\end{align*}
	is fulfilled.
\end{pro}
\begin{rem}
	 The weight $w$ for which this KP equation appears is
\begin{align*}
	w(z)&=\frac{(a_1)_z\cdots(a_M)_z}{\Gamma(z+1)(b_1)_z\cdots(b_{N})_z}\eta^z\eta_2^{z^2}\eta_3^{z^3},&
	|\eta_2|,|\eta_3|<1.
\end{align*}
and the  corresponding $0$-th moment  is the series
\begin{align*}
	\rho_0&=\sum_{k=0}^\infty w(k)=\sum_{k=0}^\infty \frac{(a_1)_k\cdots(a_M)_k}{(b_1+1)_k\cdots(b_{N-1}+1)_k}\frac{\eta^k\eta_2^{k^2}\eta_3^{k^3}}{k!}.
\end{align*}
We see that only for $\eta_2=\eta_3=1$ we recover a hypergeometric function.
\end{rem}
\subsection{Toda--Pearson compatibility}

For the compatibility of  \eqref{eq:PascalP} and \eqref{eq:MP}, that is for the compatibility of the systems
\begin{align*}
\begin{cases}
\begin{aligned}
P(z+1)&=\Pi P(z),\\
\vartheta_\eta (P(z))&=\Phi  P(z).
\end{aligned}
\end{cases}
\end{align*}
we require
\begin{align*}
\vartheta_\eta(P(z+1)
)&=\vartheta_\eta(\Pi P(z))=\vartheta_\eta(\Pi )P(z)+\Pi \vartheta_\eta (P(z))=(\vartheta_\eta(\Pi )+\Pi J\Phi )P(z),\\
(\vartheta_\eta P )(z+1)&=\Phi   P(z+1)=\Phi  \Pi  P(z)
\end{align*}
to be equal and, consequently,  we obtain
\begin{align*}
\vartheta_\eta(\Pi )=[\Phi ,\Pi ].
\end{align*}

In the general case the dressed Pascal matrix $\Pi$ is a lower unitriangular semi-infinite matrix,  that possibly has an infinite number of subdiagonals. However, for the case when the weight $w(z)=v(z)\eta^z$ satisfies the Pearson equation \eqref{eq:Pearson}, with $v$ independent of $\eta$, that is
\begin{align*}
\theta(k+1)v(k+1)\eta^{k+1}=\sigma(k)v(k)\eta^k
\end{align*}
for $k\in\N_0$; i.e.,
\begin{align*}
\theta(k+1)v(k+1)\eta=\sigma(k)v(k),
\end{align*}
the situation improves as we have the banded semi-infinite matrix $\Psi$ that models the shift in the $z$ variable as in \eqref{eq:P_shift}. From the previous discrete 
Pearson equation we see that  $\sigma(z)=\eta\kappa(z)$ with $\kappa,\theta$  $\eta$-independent polynomials in $z$
\begin{align}\label{eq:Pearson_Toda}
\theta(k+1)v(k+1)=\eta\kappa(k)v(k).
\end{align}

\begin{pro}\label{pro:compatibility Pearson Toda}
	Let us assume   a weight $w$ satisfying the Pearson equation
	\eqref{eq:Pearson} and, consequently,  of the form \eqref{eq:Pearson_weight}. Then, the Laguerre--Freud structure matrix $ \Psi$ given in \eqref{eq:Psi} satisfies

\begin{subequations}
		\begin{align}\label{eq:eta_compatibility_Pearson_1a}
	\vartheta_\eta(\eta^{-1}\Psi^\top H^{-1} )&=[\Phi ,\eta^{-1}\Psi^\top H^{-1}  ], \\
	\vartheta_\eta(\Psi H^{-1} )&=[\Phi ,\Psi H^{-1} ].\label{eq:eta_compatibility_Pearson_1b}
	\end{align}
\end{subequations}
	Alternatively, the above equations can be written as follows
\begin{subequations}
			\begin{align}\label{eq:eta_compatibility_Pearson_2a}
	\vartheta_\eta(\Psi^\top H^{-1} )&=[J_+ ,\Psi^\top H^{-1}  ], \\ \label{eq:eta_compatibility_Pearson_2b}
	\vartheta_\eta(\eta^{-1}\Psi H^{-1} )&=[J_+ ,\eta^{-1}\Psi H^{-1} ].
	\end{align}
\end{subequations}
Relations \eqref{eq:eta_compatibility_Pearson_1a} and \eqref{eq:eta_compatibility_Pearson_1b} are \emph{gauge} equivalent.
\end{pro}
\begin{proof}
 We look at the compatibility of \eqref{eq:P_shift} and \eqref{eq:MP}; i. e. for 	 \eqref{eq:eta_compatibility_Pearson_1a} we consider
	\begin{align}\label{eq:first_couple}
	\begin{cases}
	\begin{aligned}
	\sigma(z)P(z+1)&=\Psi^\top H^{-1} P(z),\\
	\vartheta_\eta (P(z))&=\Phi  P(z),
	\end{aligned}
	\end{cases}
	\end{align}
while  for  \eqref{eq:eta_compatibility_Pearson_1b} we consider
	\begin{align}\label{eq:second_couple}
	\begin{cases}
	\begin{aligned}
	\theta(z)P(z-1)&=\Psi H^{-1} P(z),\\
	\vartheta_\eta (P(z))&=\Phi  P(z).
	\end{aligned}
	\end{cases}
	\end{align}
	
	From \eqref{eq:first_couple}, observing that  $\vartheta_\eta(\sigma)=\sigma$, we deduce
	\begin{align*}
	\vartheta_\eta\big(\sigma(z)P(z+1)\big)&=\vartheta_\eta(\Psi^\top H^{-1} P(z))
	=\vartheta_\eta(\Psi^\top H^{-1}) P(z)+\Psi^\top H^{-1} \vartheta_\eta(P(z))
	\\&=\big(\vartheta_\eta(\Psi^\top H^{-1}) +\Psi^\top H^{-1} \Phi  \big)P(z),\\
	\vartheta_\eta\big(\sigma(z)P(z+1)\big)&=\sigma(z)P(z+1)+\sigma(z)\Phi  P(z+1)=(\Phi+I)  \Psi^\top H^{-1} P(z).
	\end{align*}
	so that
	\begin{align*}
\vartheta_\eta(\Psi^\top H^{-1}) -\Psi^\top H^{-1}=\big[\Phi , \Psi^\top H^{-1}\big],
	\end{align*}
and recalling $\vartheta_\eta(\eta^{-1})	=-\eta^{-1}$ we obtain 	 \eqref{eq:eta_compatibility_Pearson_1a}.
	
	From \eqref{eq:second_couple} we find ($\vartheta_\eta\theta=0$)
		\begin{align*}
	\vartheta_\eta\big(\theta(z)P(z-1)\big)&=\vartheta_\eta(\Psi H^{-1} P(z))
	=\vartheta_\eta(\Psi H^{-1}) P(z)+\Psi H^{-1} \vartheta_\eta(P(z))
	\\&=\big(\vartheta_\eta(\Psi H^{-1}) +\Psi H^{-1} \Phi  \big)P(z),\\
	\vartheta_\eta\big(\theta(z)P(z-1)\big)&=\theta(z)\Phi P(z-1)=\Phi \Psi H^{-1} P(z).
	\end{align*}
	and  \eqref{eq:eta_compatibility_Pearson_1b} follows immediately.

	Notice that from \eqref{eq:PsiBLambda} this equation \eqref{eq:eta_compatibility_Pearson_1b}  follows immediately:
	\begin{align*}
		\vartheta_{\eta}(\Psi H^{-1})=[(\vartheta_{\eta}S) S^{-1}, S B^{-1} \theta(\Lambda)S^{-1}]=[\Phi, \Psi H^{-1}].
	\end{align*}
This comment applies similarly to \eqref{eq:eta_compatibility_Pearson_1a}.
	
	Using $\Phi=-J_-=J_+-J$ in \eqref{eq:eta_compatibility_Pearson_1a} and \eqref{eq:compatibility_Jacobi_structure_a} we easily get \eqref{eq:eta_compatibility_Pearson_2a}, and similarly from\eqref{eq:eta_compatibility_Pearson_1b} and \eqref{eq:compatibility_Jacobi_structure_b} we deduce \eqref{eq:eta_compatibility_Pearson_2b}.
	
		Finally, we study the relation between \eqref{eq:eta_compatibility_Pearson_1a} and \eqref{eq:eta_compatibility_Pearson_1b}. Assume that \eqref{eq:eta_compatibility_Pearson_1b}, $\vartheta_\eta(\Psi H^{-1})=[\Phi,\Psi H^{-1}]$, holds and take its transpose to obtain  $\vartheta_\eta( H^{-1}\Psi^\top)=-[\Phi^\top, H^{-1}\Psi^\top]$, which is not \eqref{eq:eta_compatibility_Pearson_1a}. Observe that
		\begin{align*}
		\Psi^\top H^{-1}=H (H^{-1} \Psi ^\top)H^{-1}
		\end{align*}
		so that, the following \emph{gauge} type transformation equation arises
		\begin{align*}
	\vartheta_\eta(	\Psi^\top H^{-1})&=\vartheta_\eta (H) (H^{-1} \Psi ^\top)H^{-1}+
	H  \vartheta_\eta  (H^{-1} \Psi ^\top)H^{-1}-H (H^{-1} \Psi ^\top)H^{-1}\vartheta_\eta (H)H^{-1}\\&=\big[\vartheta_\eta (H) H^{-1},H (H^{-1} \Psi ^\top)H^{-1}\big]
	-H[\Phi^\top, H^{-1}\Psi^\top]H^{-1}\\&=
	\big[\vartheta_\eta (H) H^{-1}-H\Phi^\top H^{-1},\Psi ^\top H^{-1}\big].
		\end{align*}
	Now, attending to \eqref{eq:Mscr}, we have
$\vartheta_\eta (H) H^{-1}-H \Phi ^\top H^{-1}=J+\Phi$ so that 
		\begin{align*}
\vartheta_\eta(	\Psi^\top H^{-1})&=
\big[J+\Phi,\Psi ^\top H^{-1}\big]=\big[\Phi,\Psi ^\top H^{-1}\big]+\Psi ^\top H^{-1},
\end{align*}
and  \eqref{eq:eta_compatibility_Pearson_1a} follows immediately.
\end{proof}

\section{Conclusions and outlook}
The use of the Gauss--Borel factorization of the moment matrix has been throughly used by Adler and van Moerbeke in their studies of integrable systems and orthogonal polynomials, see \cite{adler_moerbeke_1,adler_moerbeke_2,adler_moerbeke_4}. We have extended these ideas and applied them in different contexts, CMV orthogonal polynomials, matrix orthogonal polynomials, multiple orthogonal polynomials  and multivariate orthogonal  \cite{am,afm,nuestro0,nuestro1,nuestro2,ariznabarreta_manas0,ariznabarreta_manas1,ariznabarreta_manas2,ariznabarreta_manas_toledano}.
For  an general overview see \cite{intro}.

In this paper,  we extended those ideas to the discrete world. In particular we applied this approach to the study of the consequences of the Pearson equation on the moment matrix and Jacobi matrices. For that description a new banded matrix is required, the Laguerre-- Freud structure matrix that encodes the Laguerre--Freud relations for  the recurrence coefficients. We have also found that the contiguous relations fulfilled generalized hypergeometric functions determining the moments of the weight described for the squared norms of the orthogonal polynomials a discrete Toda hierarchy known as Nijhoff--Capel equation, see \cite{nijhoff}. 

Further work in this direction are the study of the role of Christoffel and Geronimus transformations for the description of the mentioned contiguous relations, and the use of the Geronimus--Christoffel transformations  to characterize the shifts in the spectral independent variable of the orthogonal polynomials \cite{Manas}. In \cite{Fernandez-Irrisarri_Manas}  these ideas are applied  to generalized Charlier, Meixner, and type I  Hahn discrete orthogonal polynomials extending the results of \cite{diego,smet_vanassche, filipuk_vanassche0,filipuk_vanassche1,filipuk_vanassche2}.

For the future, we will  study the type II generalized Hahn polynomials, and extend these techniques to multiple discrete orthogonal polynomials  \cite{Arvesu_Coussment_Coussment_VanAssche} and its relations with the transformations presented in \cite{bfm}
and  quadrilateral lattices \cite{quadrilateral1,quadrilateral2},  

\section*{Acknowledgments} This work has its seed  in several inspiring conversations with Diego Dominici during a research stay at Johannes Kepler University at Linz.


\end{document}